\documentclass{article}
\usepackage[utf8]{inputenc}
\usepackage{geometry}
\usepackage{amsmath}
\usepackage{amssymb}
\usepackage{physics}

\usepackage{bm}

\usepackage{graphicx}

\usepackage{subcaption}

\usepackage[hidelinks]{hyperref} 

\usepackage{amsthm} 

\newtheorem{lemma}{Lemma}
\newtheorem{theorem}{Theorem}
\newtheorem{claim}{Claim}

\newtheorem{proposition}{Proposition}

\newenvironment{subproof}[1][\proofname]{%
	\begin{proof}[#1]%
	}{%
	\end{proof}%
}

\usepackage{bm}

\usepackage{subcaption}

\usepackage[affil-sl]{authblk}

\begin{document}

	\title{An analysis of least-squares oversampled collocation methods for compactly perturbed boundary integral equations in two dimensions}
		\author[1]{Georg Maierhofer}
	\affil[1]{Laboratoire Jacques-Louis Lions, Sorbonne University, France\\georg.maierhofer@sorbonne-universite.fr\vspace{1em}}
	\author[2]{Daan Huybrechs}	
	\affil[2]{Department of Computer Science\\
		KU Leuven\\ Belgium\\
		daan.huybrechs@kuleuven.be}	
	
	\maketitle
	
 \begin{abstract}
 	In recent work (Maierhofer \& Huybrechs, 2022, Adv.\ Comput.\ Math.), the authors showed that least-squares oversampling can improve the convergence properties of collocation methods for boundary integral equations involving operators of certain pseudo-differential form. The underlying principle is that the discrete method approximates a Bubnov--Galerkin method in a suitable sense. In the present work, we extend this analysis to the case when the integral operator is perturbed by a compact operator $\mathcal{K}$ which is continuous as a map on Sobolev spaces on the boundary, $\mathcal{K}:H^{p}\rightarrow H^{q}$ for all $p,q\in\mathbb{R}$.
 	
 	This study is complicated by the fact that both the test and trial functions in the discrete Bubnov-Galerkin orthogonality conditions are modified over the unperturbed setting. Our analysis guarantees that previous results concerning optimal convergence rates and sufficient rates of oversampling are preserved in the more general case. Indeed, for the first time, this analysis provides a complete explanation of the advantages of least-squares oversampled collocation for boundary integral formulations of the Laplace equation on arbitrary smooth Jordan curves in 2D. Our theoretical results are shown to be in very good agreement with numerical experiments.
 \end{abstract}

{\small\paragraph{Keywords} Convergence Analysis {\Large$\cdot$} Fredholm Integral Equations {\Large$\cdot$} Collocation Methods {\Large$\cdot$} \newline\begin{flushright} \vspace{-0.6cm}Compact Operators\ \ \ \ \ \ \ \ \ \ \ \ \ \ \ \ \ \ \ \ \ \ \ \ \ \ \ \ \ \ \ \ \ \ \ \ \ \ \ \ \ \ \ \ \ \ \ \ \ \ \ \ \ \ \ \ \ \ \ \ \ \ \ \ \ \ \ \ \ \ \ \ \ \ \ \ \ \ \ \ \ \ \ \ \ \ \ \ \ \ \,
	\end{flushright}
\paragraph{Mathematics Subject Classification (2020)}45B05 {\Large$\cdot$} 65N35}

\newpage

\section{Introduction}

Collocation methods are widely used for the numerical solution of integral equations \cite{hackbusch2012integral,brunner_2004}. Compared to Galerkin methods, collocation methods are easier and more efficient to implement, since they require fewer integral evaluations to compute the approximate solution. However, this comes at the expense of reduced robustness and reduced convergence rates.

Motivated by advances in approximation theory \cite{FNA1,FNA2}, recent years have seen the successful application of oversampling as a framework to improve the convergence properties of collocation methods, whilst broadly preserving their conceptual simplicity. Meanwhile, the potential of oversampling has been realized for instance in the context of Trefftz methods \cite{barnett2008stability,HUYBRECHS201992}, in hybrid numerical-asymptotic methods for high-frequency wave scattering \cite{gibbs2020} and for eigenvalue problems involving ordinary or partial differential equations \cite{hashemi2021rectangular}.

In recent work the authors provided a rigorous analysis of the convergence properties of least-squares oversampled collocation methods for two-dimensional boundary integral equations~\cite{maierhoferhuybrechs2021}. The method is least squares in the sense that more collocation points are considered than degrees of freedom, leading to a (dense) rectangular linear system, which is solved in a least squares sense. This work has included a detailed analysis which explains some of the favourable properties of oversampling as observed in the aforementioned studies. A part of this analysis concerned an asymptotic convergence result for spline-based collocation using equispaced points \cite[Thm.~3]{maierhoferhuybrechs2021}, which clearly shows a trade-off between the number of collocation points and the number of basis functions. This result was proven for integral operators of certain pseudo-differential form and follows a framework introduced by \cite{sloan1988quadrature}, \cite{sloan1989quadrature} and \cite{chandler1990} in the study of quadrature rules for Galerkin methods. The numerical experiments presented in \cite{maierhoferhuybrechs2021} suggest that the asymptotic convergence results hold true also for a wider class of integral operators. 

In the study of quadrature rules for Galerkin methods (and so-called qualocation methods), the literature offers a way to pass convergence results to compact perturbations of integral operators \cite[\S 3]{Arnold1985}. Unfortunately, this argument does not immediately apply to the oversampled collocation method. However, it serves as a basis and motivation for a novel perturbative analysis presented in this manuscript which allows us to extend the validity of the asymptotic convergence result of oversampled collocation to similar compact perturbations of the integral operator. This extends the validity of prior results from \cite{maierhoferhuybrechs2021} to a number of relevant cases. As an example in this paper we consider indirect integral formulations for boundary value problems of the Laplace equation on arbitrary smooth curves in 2D.

The present manuscript is structured as follows. In \S\ref{sec:mathematical_introduction} we introduce the mathematical framework, including the types of integral equations we aim to solve and the proposed method of solution, which is least-squares oversampled collocation. In \S\ref{sec:previous_convergence_results} we summarize previous convergence results of this method. In \S\ref{sec:examples_of_integral_operators_in_this_form} we show two examples of integral formulations for boundary value problems of Laplace's equation that involve integral operators studied in the present work. The main novel contribution of the paper follows in \S\ref{sec:convergence_analysis_of_perturbed_integral_operator}. The main result is formulated in Thm.~\ref{thm:compact_perturbation_of_discrete_bubnov_galerkin}, which extends \cite[Thm.~3]{maierhoferhuybrechs2021} to a much wider class of integral operators. A complete proof of this result is given in \S\ref{sec:convergence_analysis_of_perturbed_integral_operator}, though some of the lengthier calculations are relegated to \ref{app:proof_of_central_a_priori_estimate} and \ref{sec:app_proof_of_lemmas} for the sake of a clean presentation. Next, the predictions of the theorem are tested on several numerical examples in \S\ref{sec:numerical_experiments} and are found to match practical computations very well. Finally, we provide some concluding remarks and an outlook on potential future research in \S\ref{sec:conclusions}.

\section{The oversampled collocation method and previous results}\label{sec:mathematical_introduction}

We start with a number of definitions and an overview of notation. We consider integral equations of the form
\begin{align}\label{eqn:general_form_integral_equation}
	Vu=f,
\end{align}
where the integral operator $V:H^{p+\alpha}([0,1))\rightarrow H^{p-\alpha}([0,1))$ is a continuous isomorphism for a given $\alpha\in\mathbb{R}$ and any $p\in\mathbb{R}$. Here, we denote by $H^p:=H^{p}([0,1))$ the periodic Sobolev space of order $p\in\mathbb{R}$ on $[0,1)$, together with the norm $\|\cdot\|_s$. This Sobolev norm can be conveniently evaluated for integrable functions $g$ in terms of its Fourier coefficients $\hat{g}_m:=\int_0^1e^{-2\pi i m t}g(t)\,\dd t, \,m\in\mathbb{Z},$
\begin{align}
	\label{eqn:intro_def_boundary_sobolev_norms2}
	\|g\|_{s}^2&=\sum_{m\in\mathbb{Z}}[m]^{2s}|\hat{g}_m|^2, \quad
	\left[m\right]:=\begin{cases}1,&\quad\text{if\ }m=0,\\
		|m|,&\quad\text{if\ }m\neq 0.
	\end{cases}
\end{align}

Throughout this manuscript we will make use of commonly known properties of periodic Sobolev spaces. For further details on these spaces we refer the reader to \cite[\S5.3]{saranen2013periodic}. In the following, we denote by $I$ the periodic unit interval $I=[0,1)$. Our goal is to find an approximation to the unknown function $u:I\rightarrow\mathbb{C}$, given a known function $f\in C^{\infty}(I)$.

Our method of solution is an oversampled collocation method, i.e., a collocation method with more collocation points $M$ than degrees of freedom $N$. The approximation space is described by a periodic B-spline basis $\{ \chi_j \}_{j=1}^N$, consisting of piecewise polynomials of degree $d$. To be precise, for $N\in \mathbb{N}$, $N\geq 2$, we  define $S_N$ to be the space of $1$-periodic, $d-1$ times continuously differentiable piecewise polynomial functions of degree $d$, subordinate to the equispaced grid
\[
\{0,1/N,2/N,\dots,1-1/N\}.
\]
As collocation points we similarly use a grid $\{\frac{m}{M}\}_{m=0}^{M-1}$ of $M$ equispaced points, for $M\in\mathbb{N}$, $M\geq 2$. The oversampled collocation system is
\begin{equation}\label{eqn:system}
	A \mathbf{x} = \mathbf{b}, \qquad A_{i,j} = (V\chi_j)\left(\frac{i}{M}\right), \qquad b_i = f\left(\frac{i}{M}\right)
\end{equation}
in which $A \in \mathbb{C}^{M \times N}$ and $\mathbf{b} \in \mathbb{C}^M$.

This leads us to define the following bilinear form for functions $f,g\in C(I)$:
\begin{align*}
	\left\langle f,g\right\rangle_M:=\frac{1}{M}\sum_{m=0}^{M-1}\overline{f\left(\frac{m}{M}\right)}g\left(\frac{m}{M}\right).
\end{align*}
In the oversampled collocation method (as described in \cite[\S2 \& Eq.~(14)]{maierhoferhuybrechs2021}) we determine the approximation $u_N^{(M)}\in S_N$ of the true solution $u$ of Eq.~\eqref{eqn:general_form_integral_equation} from a least-squares solution of the overdetermined linear system Eq.~\eqref{eqn:system} that arises if we take $M\geq N$ collocation points. In terms of the bilinear form defined above, we can characterise (see \cite[Eq.~(14)]{maierhoferhuybrechs2021} and also \cite[\S4.2.2]{maierhoferthesis2021}) this $u_N^{(M)}$ as the unique element of $S_N$ satisfying the following conditions:
\begin{align*}
	\left\langle V\chi_N,Vu_N^{(M)}\right\rangle_M=\left\langle V\chi_N,f\right\rangle_M,\quad \forall\chi_N\in S_N.
\end{align*}
Equivalently, observing that $Vu=f$ for the exact solution $u$ of Eq.~\eqref{eqn:general_form_integral_equation}, we may specify $u_N^{(M)}$ as the unique element in $S_N$ satisfying
\begin{align}\label{eqn:discrete_orthogonality_conditions}
	\left\langle V\chi_N,Vu_N^{(M)}\right\rangle_M=\left\langle V\chi_N,Vu\right\rangle_M,\quad \forall\chi_N\in S_N.
\end{align}

\subsection{Previous convergence results}\label{sec:previous_convergence_results}

One way to study optimal convergence rates is based on Fourier analysis, as first introduced in the context of Galerkin methods by \cite{sloan1988quadrature,sloan1989quadrature,chandler1990}. Based on this framework, in \cite{maierhoferhuybrechs2021} an a-priori estimate was given which allows us to understand the error in the oversampled collocation method if $V$ has a certain pseudodifferential form. This form is described by the action of $V=V_0$ on an arbitrary function $g$ as
\begin{align}\label{eqn:general_pseudodifferential_form_V_0}
	V_0g(x)=C\left(\sum_{m\in\mathbb{Z}}[m]^{2\alpha}\hat{g}_me^{2\pi i m x}\right),
\end{align}
for some constant $C\in\mathbb{C}\setminus\{0\}$ and where $[m]$ is as defined in Eq.~\eqref{eqn:intro_def_boundary_sobolev_norms2}. By the equality in Eq.~\eqref{eqn:general_pseudodifferential_form_V_0} we mean that for $g\in C^{\infty}_{\mathrm{per}}(I)$ the expression holds exactly and the operator is extended to the domain $H^{p+\alpha}$ of $V_0$ by density. This means that $V_0$ maps Fourier modes to constant multiples of themselves with the factor $[m]^{2\alpha}$. In this specific form the following error estimate holds:
\begin{theorem}[Thm.~3 in {\cite{maierhoferthesis2021}}]\label{thm:convergence_rates_equispaced_grids}
	If the consistency condition $d>2\alpha$ is satisfied and $V=V_0$, as described in Eq.~\eqref{eqn:general_pseudodifferential_form_V_0}, then there are constants $N_0,C_{d,\alpha}>0$ depending on $d,\alpha$, but independent of $u,N,M$, such that the solution of Eq.~\eqref{eqn:discrete_orthogonality_conditions} satisfies following error estimate for any $M\geq N\geq N_0$:
	\begin{align*}
		{\|u_N^{(M)}-\tilde{u}\|_{{4\alpha-(d+1)}}}\leq C_{d,\alpha}\left(M^{-(d+1)+2\alpha}+N^{-2(d+1)+4\alpha}\right){\|\tilde{u}\|_{{d+1}}}.
	\end{align*}
\end{theorem}

This theorem allows us to draw conclusions on the amount of oversampling required to achieve optimal convergence rates with collocation methods. In particular, it highlights that the fastest rate of convergence is achieved when $M=M(N)=N^2$, and this rate is faster than the conventional Galerkin method for Eq.~\eqref{eqn:general_form_integral_equation}. Of course this is an expensive regime in practice, but even with just linear oversampling $M \sim N$, the theorem shows a clear reduction of the error by a constant factor depending on $d$, $\alpha$ and describes the proportionality constant as we also see in the numerical examples in \S\ref{sec:numerical_experiments}.

Unfortunately, the form Eq.~\eqref{eqn:general_pseudodifferential_form_V_0} is rarely sufficient to explain practical applications. Instead, several important problems can be formulated in terms of integral equations where $V$ takes the following form
\begin{align}\label{eqn:general_form_of_perturbed_integral_operator}
	V=V_0+\mathcal{K},
\end{align}
with $V_0$ as described in Eq.~\eqref{eqn:general_pseudodifferential_form_V_0} and $\mathcal{K}:H^{p}\rightarrow H^{q}$ being continuous for all $p,q\in\mathbb{R}$. We describe two such examples of integral equations next in \S\ref{sec:examples_of_integral_operators_in_this_form}. Our aim in this paper is to extend Thm.~\ref{thm:convergence_rates_equispaced_grids} to integral operators of the form Eq.~\eqref{eqn:general_form_of_perturbed_integral_operator}.

\subsection{Examples of integral operators in the form $V_0+\mathcal{K}$}\label{sec:examples_of_integral_operators_in_this_form}

We summarize some well-known integral formulations which motivate our work. For further details the reader is referred to \cite{sloan_1992} and references therein.

The integral equations described in Eqs.~\eqref{eqn:general_form_integral_equation}-\eqref{eqn:general_form_of_perturbed_integral_operator} appear for instance in indirect methods for the integral formulation of the Dirichlet and Neumann Problem of Laplace's equation. Let $\Omega\subset\mathbb{R}^2$ be a connected open subset, such that $\partial\Omega$ is a $C^\infty$ closed Jordan curve parametrised by $z$, where $z:I\rightarrow \partial\Omega\subset\mathbb{R}^2$ is bijective, infinitely differentiable and with $z'(t)\neq0, \forall t\in I=[0,1)$. Consider the following two boundary value problems: the interior Dirichlet problem for the Laplace equation
\begin{align}\label{eqn:interior_dirichlet_Laplace}
	\begin{cases}
		\Delta \phi=0, &x\in\Omega, \\
		\phi=f, &x\in \partial\Omega,\\
	\end{cases}
\end{align}
and the exterior Neumann problem for the Laplace equation
\begin{align}\label{eqn:exterior_neumann_Laplace}
	\begin{cases}
		\Delta \phi=0, &x\in \mathbb{R}^2\setminus\Omega ,\\
		\partial_n\phi=g, &x\in \partial\Omega,\\
		\phi(x)\rightarrow 0, &\quad |x|\rightarrow\infty.
	\end{cases}
\end{align}

Both of these can be formulated in terms of integral equations of the form Eqs.~\eqref{eqn:general_form_integral_equation}-\eqref{eqn:general_form_of_perturbed_integral_operator} \cite[\S 2]{sloan_1992}. For Eq.~\eqref{eqn:interior_dirichlet_Laplace} we may seek to express
\begin{align}\label{eqn:phi_single_layer_expression}
	\phi(x)=\frac{1}{2\pi}\int_0^1\log|z(t)-x|\, v(t)\,|z'(t)|\dd t,\quad \forall x\in\Omega,
\end{align}
where $v:I\rightarrow \mathbb{C}$ is the so-called `single-layer density'. It is given as the unique function that satisfies
\begin{align}\label{eqn:single_layer_formulation_interior_Dirichlet}
	\int_0^1\frac{1}{2\pi}\log|z(t)-z(s)|\, v(t)\,|z'(t)|\,\dd t=f(z(s)), \quad s\in I.
\end{align}
For the exterior Neuman problem Eq.~\eqref{eqn:general_form_of_perturbed_integral_operator} a common formulation is to express $\phi$ again in the form Eq.~\eqref{eqn:phi_single_layer_expression}, where $u:I\rightarrow \mathbb{C}$ is the unique solution of the following integral equation:
\begin{align}\label{eqn:adjoint_double_layer_formulation_exterior_neumann}
	-\frac{1}{2}v(s) +\frac{1}{2\pi}\int_0^1\frac{n(s)\cdot(z(t)-z(s))}{|z(t)-z(s)|^2} \,v(t)\,|z'(t)|\,\dd t=g(z(s)),\quad\forall s\in I.
\end{align}
Here, $n:I\rightarrow \mathbb{C}^2$ is the unit normal vector to $\partial\Omega$ which points \textit{outward} of $\Omega$. Let us define the new unknown $u(t):=v(t)\,|z'(t)|$. Then Eq.~\eqref{eqn:single_layer_formulation_interior_Dirichlet} is of the form
\begin{align*}
	\mathcal{S}u=f\circ z,
\end{align*}
with $\mathcal{S}$ the single layer operator
\begin{align*}
	\mathcal{S}u (s)=\int_0^1\frac{1}{2\pi}\log|z(t)-z(s)|u(t)\, \dd t.
\end{align*}

It can be seen that the operator $\mathcal{S}$ for an arbitrary $C^\infty$ Jordan curve $\partial\Omega$ behaves largely like the corresponding operator on a circle \cite[pp. 299-300]{sloan_1992}. By extracting this dominant weakly singular part of the kernel in appropriate form (cf. \cite[Eqs.~(3.8)-(3.9)]{sloan_1992}) one then finds that $\mathcal{S}$ takes the form
\begin{align*}
	\mathcal{S}=\mathcal{S}_0+\mathcal{K}_1,
\end{align*}
where $\mathcal{S}_0g(x)=\sum_{m\in\mathbb{Z}}[m]^{-1}\hat{g}_me^{2\pi i m x}$ and $\mathcal{K}_1$ is an integral operator with a $C^{\infty}$ kernel function, hence is continuous as $\mathcal{K}_1:H^{s}\rightarrow H^t$ for any $s,t\in\mathbb{R}$. This means that $\mathcal{S}$ is precisely of the form Eq.~\eqref{eqn:general_form_of_perturbed_integral_operator}. Moreover, it is known that $\mathcal{S}:H^{s-1/2}\rightarrow H^{s+1/2}$ is an isomorphism for any $s\in\mathbb{R}$ so long as the transfinite diameter of the curve $\partial \Omega$ does not equal 1 \cite[Eq.~(4.36)]{sloan_1992}. For more details on this condition we refer the reader to \cite[p. 306]{sloan_1992} and we highlight that in the cases considered in the numerical experiments in \S\ref{sec:numerical_experiments} we have ensured that the transfinite diameter of the relevant boundary is not equal to $1$.

Let us now consider Eq.~\eqref{eqn:adjoint_double_layer_formulation_exterior_neumann}. Defining $u(t):=v(t)\,|z'(t)|$ we can write the equation in the form
\begin{align*}
	\left(-\frac{1}{2}\mathcal{I}+\mathcal{D}^*\right)v(s)=\frac{1}{|z'(s)|}g\circ z(s),\quad\forall s\in I,
\end{align*}
where $\mathcal{I}$ is the identity map and
\begin{align*}
	\mathcal{D}^*u(s)=\frac{1}{2\pi}\int_0^1\frac{n(s)\cdot(z(t)-z(s))}{|z'(s)||z(t)-z(s)|^2} \,u(t)\,\dd t,\quad\forall s\in I.
\end{align*}
We want to demonstrate that $-\frac{1}{2}\mathcal{I}+\mathcal{D}^*$ itself already has the required form of Eq.\eqref{eqn:general_form_of_perturbed_integral_operator}. To that end, let us consider the kernel
\begin{align*}
	k:(s,t)\mapsto \frac{n(s)\cdot(z(t)-z(s))}{|z(t)-z(s)|^2}.
\end{align*}
Because $z$ is infinitely differentiable so is $k$ for any $s\neq t$. By definition of the normal to $\partial\Omega$ we have $n(s)\cdot z'(s)=0$ for all $s\in I$, thus we also have, using Taylor's theorem
\begin{align*}
	n(s)\cdot(z(t)-z(s))=n(s)\cdot((t-s)z'(s)+(t-s)^2a(t,s))=(t-s)^2n(s)\cdot a(t),
\end{align*}
where $a:I\times I\rightarrow \mathbb{R}^2$ is infinitely differentiable. Similarly we have
\begin{align*}
	|z(t)-z(s)|^2=(t-s)^2|z'(t)+(t-s)b(t,s)|^2,
\end{align*}
where $b:I\times I\rightarrow \mathbb{R}^2$ is infinitely differentiable. Since $z'(t)\neq0$ it immediately follows that $k$ is also infinitely differentiable at $s=t$, i.e. $k\in C^\infty(I\times I)$. Thus $\mathcal{D}^*$ is an integral operator with smooth kernel function and hence $\mathcal{D}^*:H^{s}\rightarrow H^t$ is continuous for all $s,t\in\mathbb{R}$. Moreover $-\mathcal{I}/2$ is clearly of the form Eq.~\eqref{eqn:general_pseudodifferential_form_V_0} and so overall $-\mathcal{I}/{2}+\mathcal{D}^*$ takes the form Eq.~\eqref{eqn:general_form_of_perturbed_integral_operator}. Finally, we note that $-\mathcal{I}{2}+\mathcal{D}^*:H^s\rightarrow H^s$ is a continuous isomorphism for any $s\in\mathbb{R}$ \cite[p. 303]{sloan_1992} (see also \cite[\S 13]{Mikhlin1970}).


		\section{Convergence analysis for a compactly perturbed integral operator}\label{sec:convergence_analysis_of_perturbed_integral_operator}
		\subsection{Preliminaries and properties of the perturbed system}
		We will now seek to derive a similar estimate to Thm.~\ref{thm:convergence_rates_equispaced_grids} under the assumption that $V$ takes the form $V=V_0+\mathcal{K}$ as introduced in Eqs.~\eqref{eqn:general_pseudodifferential_form_V_0}-\eqref{eqn:general_form_of_perturbed_integral_operator}. It will be helpful to write this perturbation of $V_0$ in the following form
		\begin{align*}
			V=(\mathcal{I}+\tilde{\mathcal{K}})V_0
		\end{align*}
		where $V_0$ is as defined in Eq.~\eqref{eqn:general_pseudodifferential_form_V_0}, $\mathcal{I}+\tilde{\mathcal{K}}:H^{p}\rightarrow H^{p}$ is a continuous isomorphism for all $p\in\mathbb{R}$ and $\tilde{\mathcal{K}}:H^p\rightarrow H^q$ is continuous for any $p,q\in\mathbb{R}$. We arrive at this form simply by defining
		\begin{align*}
			\tilde{\mathcal{K}}:=\mathcal{K}V_0^{-1}.
		\end{align*}
		since, by the pseudodifferential form of $V_0$, the map $\mathcal{K}V_0^{-1}:H^{s}\rightarrow H^{t}$ is still continuous for all $s,t\in\mathbb{R}$, and $V:H^{s+2\alpha}\rightarrow H^{s}$ being invertible for all $s\in\mathbb{R}$ is equivalent to $\mathcal{I}+\mathcal{K}V_0^{-1}:H^{s}\rightarrow H^{s}$ being invertible for all $s\in\mathbb{R}$.
		
		For notational simplicity we will henceforth write $\mathcal{K}$ instead of $\tilde{\mathcal{K}}$. Note that the continuity properties of $\mathcal{K}$ allow us to represent the map $\mathcal{K}$ by its action on the Fourier basis, i.e.\ letting $k_{mn}=\langle\exp(2\pi i n\,\cdot\,),\mathcal{K}\exp(2\pi im\,\cdot\,)\rangle_{L^2}$ we have for any $u\in L^2$
		\begin{align*}
			(\mathcal{K}u)(x)=\sum_{m\in\mathbb{Z}}\sum_{n\in\mathbb{Z}}k_{mn}\hat{u}_me^{2\pi i n x},
		\end{align*}
		and the series converges absolutely uniformly, since (by continuity of $\mathcal{K}:H^{p}\rightarrow H^{q},\,\forall p,q\in\mathbb{R}$) for every $p,q\in\mathbb{R}$ there is $C_{p,q}>0$ such that
		\begin{align}\label{eqn:decay_of_K_coefficients}
			|k_{mn}|\leq C_{p,q}(1+|m|)^{-p}(1+|n|)^{-q}, \quad \forall\, m,n\in\mathbb{Z}.
		\end{align}
		Similarly, $\mathcal{K}^*$ is represented by the conjugate transpose of these values, i.e. 
		\begin{align*}
			\langle\exp(2\pi i n\,\cdot\,),\mathcal{K}^*\exp(2\pi im\,\cdot\,)\rangle_{L^2}=\overline{k_{nm}}
		\end{align*}
		where $\overline{y}$ denotes the complex conjugate of $y\in\mathbb{C}$.
		
		\subsection{The perturbed orthogonality conditions and convergence result}
		Thus for an integral operator $V$ of the form Eq.~\eqref{eqn:general_form_of_perturbed_integral_operator} the oversampled collocation method gives rise to the following set of discrete orthogonality conditions which uniquely determines $u_N^{(M)}$: $\forall\chi_N\in S_N$
		\begin{align}\label{eqn:discrete_bubnov_compact_perturbation_orthogonality_conditions}
			\left\langle (\mathcal{I}+\mathcal{K})V_0\chi_N,(\mathcal{I}+\mathcal{K})V_0u_N^{(M)}\right\rangle_M=	\left\langle (\mathcal{I}+\mathcal{K})V_0\chi_N, (\mathcal{I}+\mathcal{K})V_0{u}\right\rangle_M,
		\end{align}
		where for notational simplicity we again wrote $\mathcal{K}$ instead of $\tilde{\mathcal{K}}$. Ultimately we will prove the following result, which extends the conclusions of Thm.~\ref{thm:convergence_rates_equispaced_grids} to certain compact perturbations:
		\begin{theorem}\label{thm:compact_perturbation_of_discrete_bubnov_galerkin}
			If $u_N^{(M)}\in S_N$ is such that $\forall \chi_N\in S_N$
			\begin{align*}
				\left\langle (\mathcal{I}+\mathcal{K})V_0\chi_N,(\mathcal{I}+\mathcal{K})V_0u_N^{(M)}\right\rangle_M=	\left\langle (\mathcal{I}+\mathcal{K})V_0\chi_N, (\mathcal{I}+\mathcal{K})V_0{u}\right\rangle_M,
			\end{align*}
			where $V_0$ takes the form Eq.~\eqref{eqn:general_pseudodifferential_form_V_0}, $\mathcal{I}+\mathcal{K}:H^p\rightarrow H^p$ is a continuous isomorphism for all $p\in\mathbb{R}$ and $\mathcal{K}:H^{p}\rightarrow H^{q}$ is continuous for all $p,q\in\mathbb{R}$, then there are constants $N_0,C>0$ independent of $N,{u},M$ such that, for all $M\geq N\geq N_0$,
			\begin{align}\label{eqn4:full_error_estimate_final_goal}
				\|u_N^{(M)}-{u}\|_{2\alpha-(d+1)}\leq C(M^{2\alpha-(d+1)}+N^{4\alpha-2(d+1)})\|{u}\|_{d+1}.
			\end{align}
		\end{theorem}
			
			The literature (cf. \cite[\S3]{Arnold1985}) offers a standard procedure to extend asymptotic error estimates of the form in Thm.~\ref{thm:convergence_rates_equispaced_grids} to the case when only the right hand side of the integral operator in the orthogonality conditions is perturbed, i.e. if $\tilde{u}_N^{(M)}$ would satisfy
			\begin{align*}
				\left\langle V_0\chi_N,(\mathcal{I}+\mathcal{K})V_0\tilde{u}_N^{(M)}\right\rangle_M=	\left\langle V_0\chi_N, (\mathcal{I}+\mathcal{K})V_0{u}\right\rangle_M,\quad\forall \chi_N\in S_N.
			\end{align*}
			A detailed description of the argument can also be found in \cite[Appendix~D]{maierhoferhuybrechs2021}. Thus it is suggestive to attempt to find a way to take the `discrete adjoint' of $\mathcal{K}$ with respect to $\langle\cdot,\cdot\rangle_M$. Specifically, we would like to formulate the orthogonality conditions Eq.~\eqref{eqn:discrete_bubnov_compact_perturbation_orthogonality_conditions} in a form similar to
			\begin{align}\label{eqn:wishful_discrete_adjoint_equations}
				\left\langle V_0\chi_N,(\mathcal{I}+\mathcal{K}^*)(\mathcal{I}+\mathcal{K})V_0u_N^{(M)}\right\rangle_M \stackrel{?}{=}	\left\langle V_0\chi_N, (\mathcal{I}+\mathcal{K}^*)(\mathcal{I}+\mathcal{K})V_0{u}\right\rangle_M,
			\end{align}
			$\forall \chi_N\in S_N,$ where by $\mathcal{K}^*$ we have denoted the continuous adjoint map corresponding to $\mathcal{K}$, which is a continuous map $\mathcal{K}^*:H^{-q}\rightarrow H^{-p}$ for all $p,q\in\mathbb{R}$. We note that Eq.~\eqref{eqn:wishful_discrete_adjoint_equations} would be exactly equivalent to Eq.~\eqref{eqn:discrete_bubnov_compact_perturbation_orthogonality_conditions} if we were to replace $\langle\cdot,\cdot\rangle_M$ by the exact $L^2$-inner product $\langle\cdot,\cdot\rangle_{L^2}$. However, the discrete nature of $\langle\cdot,\cdot\rangle_M$ prevents this exact equivalence, and so we need to find a way to account for the error incurred in $u_{N}^{(M)}$ when we choose to solve Eq.~\eqref{eqn:wishful_discrete_adjoint_equations} instead of Eq.~\eqref{eqn:discrete_bubnov_compact_perturbation_orthogonality_conditions}.
			
			In order to do so let us introduce the following bilinear form $\epsilon$:
			\begin{align*}
				\epsilon(\chi_N,\tilde{b}):=\left\langle V_0\chi_N,\mathcal{K}^* V_0\tilde{b}\right\rangle_M-\left\langle \mathcal{K}V_0\chi_N,V_0\tilde{b}\right\rangle_M.
			\end{align*}
			
			Using this bilinear form we can reformulate the full perturbed orthogonality conditions Eq.~\eqref{eqn:discrete_bubnov_compact_perturbation_orthogonality_conditions} in the following equivalent form: $\forall\chi_N\in S_N$
			\begin{align*}
				\left\langle V_0\chi_N,(\mathcal{I}+\mathcal{K}^*)(\mathcal{I}+\mathcal{K})V_0u_N^{(M)}\right\rangle_M&= \left\langle V_0\chi_N,(\mathcal{I}+\mathcal{K}^*)(\mathcal{I}+\mathcal{K})V_0u\right\rangle_M\\
				&\quad+\epsilon\left(\chi_N,V_0^{-1}(\mathcal{I}+\mathcal{K})V_0\left(u_N^{(M)}-u\right)\right).
			\end{align*}
			\subsection{Proof of the new convergence result for the perturbed operator}
			With this formulation we can now try to account for the error incurred when the orthogonality conditions are perturbed using $\epsilon$. The following is the central new a-priori estimate facilitating the proof of Thm.~\ref{thm:compact_perturbation_of_discrete_bubnov_galerkin}.
			\begin{proposition}\label{prop:perturbed_orthogonality_conditions}
				Suppose $a_N^{(M)}\in S_N$ satisfies 
				\begin{align}\label{eqn:unperturbed_apriori_orhtogonality_condition1}
					\left\langle V_0\chi_N,V_0a_N^{(M)}\right\rangle_M=\left\langle V_0\chi_N,V_0\tilde{v}\right\rangle_M+\epsilon(\chi_N,b-c_N),\ \quad\forall \chi_N\in S_N,
				\end{align}
				for some $\tilde{v},b\in H^{d+1}$ and a sequence of smoothest splines $c_N\in S_N$, $N\in\mathbb{N}$. Then there is a constant $C>0$ independent of $N,M,\tilde{v},b,c_N$ such that for all $M\geq N>0$:
				\begin{align}\begin{split}\label{eqn:proposition_central_estimate}
						\|\tilde{v}-a_N^{(M)}\|_{4\alpha-(d+1)}&\leq C \left(M^{2\alpha-(d+1)}+N^{4\alpha-2(d+1)}\right)\left(\|\tilde{v}\|_{d+1}+\|b\|_{d+1}\right)\\
						&\quad\quad +CN^{-1}\|b-c_N\|_{4\alpha-(d+1)}.\end{split}
				\end{align}
			\end{proposition}
			\begin{proof}
				The proof of this estimate is presented in \ref{app:proof_of_central_a_priori_estimate}. It relies essentially on the observation that the discrete bilinear form $\left\langle\,\cdot\,,\,\cdot\,\right\rangle_M$ results in aliasing whereby the low-frequency terms correspond to the exact $L^2$-inner product on $I$ and hence cancel exactly in the contribution of the bilinear form $\epsilon(\,\cdot\,,\,\cdot\,)$. The high-frequency terms can be bounded, by exploiting the continuity properties of the operator $\mathcal{K}$, by the term $CN^{-1}\|b-c_N\|_{4\alpha-(d+1)}$ resulting in the estimate Eq.~\eqref{eqn:proposition_central_estimate}.
			\end{proof}
			
			We can now use Prop.~\ref{prop:perturbed_orthogonality_conditions} to prove Thm.~\ref{thm:compact_perturbation_of_discrete_bubnov_galerkin} in a manner similar to the perturbation argument of \cite[\S3]{Arnold1985}.
			
			\begin{proof}[Proof of Thm.~\ref{thm:compact_perturbation_of_discrete_bubnov_galerkin}]
				We proceed in two steps: Firstly we show that a perturbation of the test functions $V_0\chi_N\mapsto (\mathcal{I}+\mathcal{K})V_0\chi_N$ yields a similar error estimate as in Thm.~\ref{thm:convergence_rates_equispaced_grids} and then we proceed to perturb the operator $V_0$ on the right hand side of the orthogonality conditions.
				\begin{claim}\label{claim:discrete_orthogonality_conditions_perturbed_test_functions}
					Suppose $a_N^{(M)}\in S_N$ satisfies
					\begin{align}\label{eqn:discrete_orthogonality_conditions_perturbed_test_functions}
						\left\langle (\mathcal{I}+\mathcal{K})V_0\chi_N,V_0a_N^{(M)}\right\rangle_M=\left\langle (\mathcal{I}+\mathcal{K})V_0\chi_N,V_0\tilde{a}\right\rangle_M,\quad \forall \chi_N\in S_N,
					\end{align}
					where $\mathcal{K}$ satisfies the assumptions of Thm.~\ref{thm:compact_perturbation_of_discrete_bubnov_galerkin}. Then there are constants $C, N_0>0$ independent of $\tilde{a},a_N^{(M)},M,N$ such that for $M\geq N\geq N_0$:
					\begin{align}\label{eqn:discrete_orthogonality_conditions_perturbed_test_functions_estimate}
						\|a_N^{(M)}-\tilde{a}\|_{4\alpha-(d+1)}&\leq C(M^{2\alpha-(d+1)}+N^{4\alpha-2(d+1)})\|\tilde{a}\|_{d+1}.
					\end{align}
				\end{claim}
				\begin{subproof}[Proof of Claim~\ref{claim:discrete_orthogonality_conditions_perturbed_test_functions}]
					To begin with we note that Eq.~\eqref{eqn:discrete_orthogonality_conditions_perturbed_test_functions} is equivalent to: $\forall\chi_N\in S_N$:
					\begin{align*}
						\left\langle V_0\chi_N,(\mathcal{I}+\mathcal{K}^*)V_0a_N^{(M)}\right\rangle_M=\left\langle V_0\chi_N,(\mathcal{I}+\mathcal{K}^*)V_0\tilde{a}\right\rangle_M+\epsilon(\chi_N,a_N^{(M)}-\tilde{a}).
					\end{align*}
					This can be equivalently written as: $\forall \chi_N\in S_N$:
					\begin{align*}
						\left\langle\hspace{-0.05cm} V_0\chi_N,V_0a_N^{(M)}\hspace{-0.05cm}\right\rangle_M\hspace{-0.1cm}&=\hspace{-0.025cm}\left\langle V_0\chi_N,V_0\left(\tilde{a}+V_0^{-1}\mathcal{K}^* V_0(\tilde{a}-a_N^{(M)})\right)\right\rangle_M\hspace{-0.075cm}+\epsilon(\chi_N,a_N^{(M)}-\tilde{a}).
					\end{align*}
					Therefore, Prop.~\ref{prop:perturbed_orthogonality_conditions} applies and shows that
					\begin{align*}
						\|V_0^{-1}(\mathcal{I}+\mathcal{K}^*)& V_0(\tilde{a}-a_N^{(M)})\|_{4\alpha-(d+1)}\\&\leq C\left(M^{2\alpha-(d+1)}+N^{4\alpha-2(d+1)}\right)\|\tilde{a}+V_0^{-1}\mathcal{K}^* V_0(\tilde{a}-a_N^{(M)})\|_{d+1}\\
						&\quad\quad +C\left(M^{2\alpha-(d+1)}+N^{4\alpha-2(d+1)}\right)\|\tilde{a}\|_{d+1}\\
						&\quad\quad +CN^{-1}\|\tilde{a}-a_N^{(M)}\|_{4\alpha-(d+1)}.
					\end{align*}
					Noting that $V_0:H^{4\alpha-(d+1)}\rightarrow H^{2\alpha-(d+1)}$ is a continuous isomorphism, that $\mathcal{I}+\mathcal{K}^*:H^{2\alpha-(d+1)}\rightarrow H^{2\alpha-(d+1)}$ is invertible, and that $V_0^{-1}\mathcal{K}^*V_0:H^{4\alpha-(d+1)}\rightarrow H^{d+1}$ is bounded (by the assumptions on $\mathcal{K}$), there is a constant $\tilde{C}>0$ such that
					\begin{align*}
						\|\tilde{a}-a_N^{(M)}\|_{4\alpha-(d+1)}&\leq \tilde{C}\left(M^{2\alpha-(d+1)}+N^{4\alpha-2(d+1)}\right)\|\tilde{a}\|_{d+1}\\
						&+\tilde{C}\left(M^{2\alpha-(d+1)}+N^{4\alpha-2(d+1)}+N^{-1}\right)\|\tilde{a}-a_N^{(M)}\|_{4\alpha-(d+1)}.
					\end{align*}
					{Equivalently,
						\begin{align*}
							\left(1-\tilde{C}\left(M^{2\alpha-(d+1)}+N^{4\alpha-2(d+1)}+N^{-1}\right)\right)&\|\tilde{a}-a_N^{(M)}\|_{4\alpha-(d+1)}\\
							&\hspace{-1cm}\leq \tilde{C}\left(M^{2\alpha-(d+1)}+N^{4\alpha-2(d+1)}\right)\|\tilde{a}\|_{d+1}.
					\end{align*}}
					Thus we conclude that, for $M,N$ sufficiently large, the estimate Eq.~\eqref{eqn:discrete_orthogonality_conditions_perturbed_test_functions_estimate} holds.
				\end{subproof}
				Having proved Claim~\ref{claim:discrete_orthogonality_conditions_perturbed_test_functions} we can proceed to prove Thm.~\ref{thm:compact_perturbation_of_discrete_bubnov_galerkin} as follows. Suppose $u_N^{(M)}\in S_N$ satisfies: $\forall \chi_N\in S_N$,
				\begin{align*}
					\left\langle (\mathcal{I}+\mathcal{K})V_0\chi_N,(\mathcal{I}+\mathcal{K})V_0u_N^{(M)}\right\rangle_M=	\left\langle (\mathcal{I}+\mathcal{K})V_0\chi_N, (\mathcal{I}+\mathcal{K})V_0{u}\right\rangle_M.
				\end{align*}
				These conditions are equivalent to: $\forall\chi_N\in S_N$,
				\begin{align*}
					\left\langle (\mathcal{I}+\mathcal{K})V_0\chi_N,V_0u_N^{(M)}\right\rangle_M=\left\langle (\mathcal{I}+\mathcal{K})V_0\chi_N, V_0\left({u}+V_0^{-1}\mathcal{K}V_0({u}-u_N^{(M)})\right)\right\rangle_M.
				\end{align*}
				Thus by Claim~\ref{claim:discrete_orthogonality_conditions_perturbed_test_functions} we have for some $C>0$
				\begin{align*}
					\!\|V_0^{-1}(\mathcal{I}+\mathcal{K})V_0&\!\left(u_N^{(M)}-{u}\right)\!\|_{4\alpha-(d+1)}\!\leq\! C\!\left(M^{2\alpha-(d+1)}+N^{4\alpha-2(d+1)}\right)\!\|{u}+V_0^{-1}\mathcal{K}V_0({u}-u_N^{(M)})\|_{d+1}.
				\end{align*}
				We note that by continuity of $\mathcal{K}:H^{2\alpha-(d+1)}\rightarrow H^{d+1-2\alpha}$ we have, for some $C_2>0$,
				\begin{align*}
					\|{u}+V_0^{-1}\mathcal{K}V_0({u}-u_N^{(M)})\|_{d+1}\leq \|{u}\|_{d+1}+C_2\|{u}-u_N^{(M)}\|_{4\alpha-(d+1)}.
				\end{align*}
				Moreover, by the assumptions on $V_0,\mathcal{K}$ the map $V_0^{-1}(\mathcal{I}+\mathcal{K})^{-1}V_0:H^{4\alpha-(d+1)}\rightarrow H^{4\alpha-(d+1)}$ is bounded and, therefore, we have, for some $\tilde{C}>0$ independent of ${u},u^{(M)}_N,M,N$,
				\begin{align*}
					\left(1-\tilde{C}\left(M^{2\alpha-(d+1)}+N^{4\alpha-2(d+1)}\right)\right)&\|{u}-u_N^{(M)}\|_{4\alpha-(d+1)}\!\leq\!\tilde{C} \!\left(M^{2\alpha-(d+1)}+N^{4\alpha-2(d+1)}\right)\!\|{u}\|_{d+1}.
				\end{align*}
				Thus, we conclude for $N,M$ sufficiently large the estimate Eq.~\eqref{eqn4:full_error_estimate_final_goal} holds, hence completing the proof of Thm.~\ref{thm:compact_perturbation_of_discrete_bubnov_galerkin}.
			\end{proof}
			\section{Numerical examples}\label{sec:numerical_experiments}
			Having proved Thm.~\ref{thm:compact_perturbation_of_discrete_bubnov_galerkin} we will now see in two numerical examples that the convergence rates predicted in Thm.~\ref{thm:compact_perturbation_of_discrete_bubnov_galerkin} are indeed observed in practice. In order to do so we apply the oversampled collocation method as introduced in \S\ref{sec:mathematical_introduction} to integral formulations of Laplace's equation as described in \S\ref{sec:examples_of_integral_operators_in_this_form}. In both numerical examples we used the Julia package \cite{simpleintegralequations} as an implementation of the relevant integral equation and numerical schemes and our reference solution is a Galerkin approximation with linear spline basis functions and an equispaced mesh with $N=4096$ points.
			\subsection{Application to potential flow about compact body}\label{sec:example_application_to_potential_flow}
			In our first example we aim to solve for the inviscid irrotational incompressible flow around a compact obstacle, $\Omega$ which in our case is the ellipse shown in Fig.~\ref{fig:obstacle_and_flow}. This means \cite[Chapter 6]{batchelor_2000} that we seek $\mathbf{v}$ such that
			\begin{align*}
				\begin{cases}\bm{\nabla}\cdot\mathbf{v}(x)=0 \,\land\,\bm{\nabla}\times\mathbf{v}(x)=0,&x\in\Omega,\\
					\mathbf{n}(x)\cdot \mathbf{v}(x)=0,& x\in\partial \Omega,\\
					\mathbf{v}(x)\rightarrow (U,0),& |x|\rightarrow \infty.
				\end{cases}
			\end{align*}
			Thus we can write $\mathbf{v}=(U,0)+\nabla\phi$, where the perturbation velocity potential satisfies Eq.~\eqref{eqn:exterior_neumann_Laplace} with $g(x)=-\mathbf{n}(x)\cdot (U,0),\, x\in\partial\Omega$. Therefore we can use Eq.~\eqref{eqn:adjoint_double_layer_formulation_exterior_neumann} to solve for $\phi$.
			\begin{figure}[h!]
				\centering\vspace{-0.1cm}
				\includegraphics[width=0.6\textwidth]{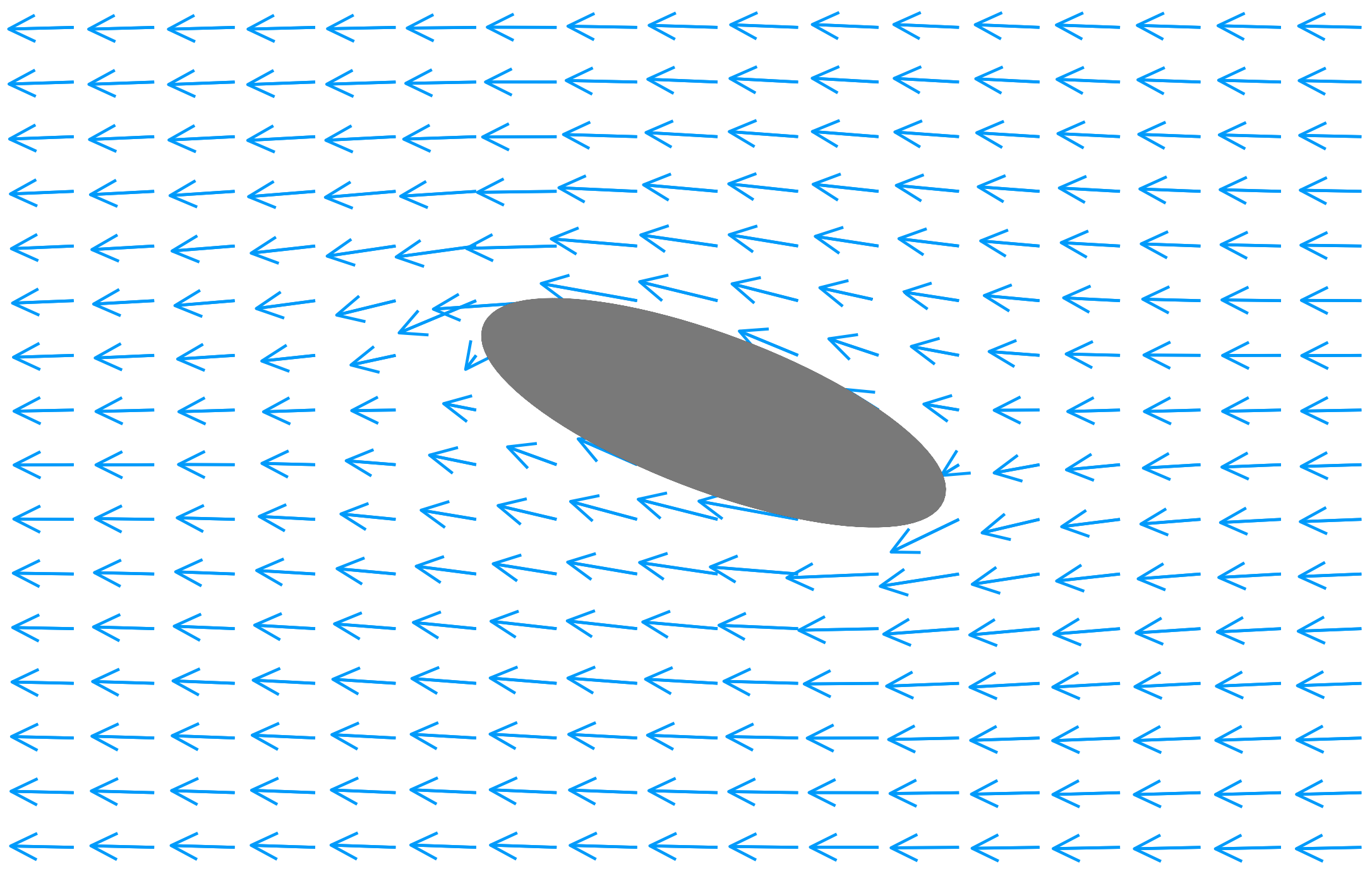}
				\caption{The velocity field $\mathbf{v}=(U,0)+\bm{\nabla}\phi$ of potential flow around $\Omega$.}
				\label{fig:obstacle_and_flow}
			\end{figure}
			
			We use a linear spline basis (i.e. $d=1$) and note that for the integral operator $V=-\mathcal{I}/2+\mathcal{D}^*$ we have $\alpha=0$. Furthermore in our example we took $U=-1$. Therefore, the results in Thm.~\ref{thm:compact_perturbation_of_discrete_bubnov_galerkin} predict the following asymptotic convergence rate for $M,N$ sufficiently large:
			\begin{align}\label{eqn:estimate_for_potential_flow_example}
				\|u_N^{(M)}-u\|_{-2}\leq C \left(M^{-2}+N^{-4}\right)\|u\|_2,
			\end{align}
			for some $C>0$. We have plotted the approximation error $ \|u_N^{(M)}-u\|_{-2}$ in Figs.~\ref{fig:Sobolev_error_flow_example_superlinear_oversampling} \& \ref{fig:Sobolev_error_flow_example_linear_oversampling}.

			In Fig.~\ref{fig:Sobolev_error_flow_example_superlinear_oversampling} we observe the asymptotic convergence rates of the overall method for the standard collocation method $M(N)=N$, the oversampled collocation method with linear oversampling $M(N)=5N$ and the oversampled collocation method with quadratic oversampling $M(N)=N^2$. The asymptotic convergence rates are indicated using the dash-dotted lines whereby here and in all following figures the constants $C_1,C_2$ bear no relation to the constants in Eq.~\eqref{eqn:estimate_for_potential_flow_example}, these are simply included so that these lines are easier to see. It can be seen that the predicted convergence rates from Eq.~\eqref{eqn:estimate_for_potential_flow_example} are exactly realised in practice.
			
			\begin{figure}[h!]
				\centering\vspace{-0.1cm}
				\includegraphics[width=0.7\textwidth]{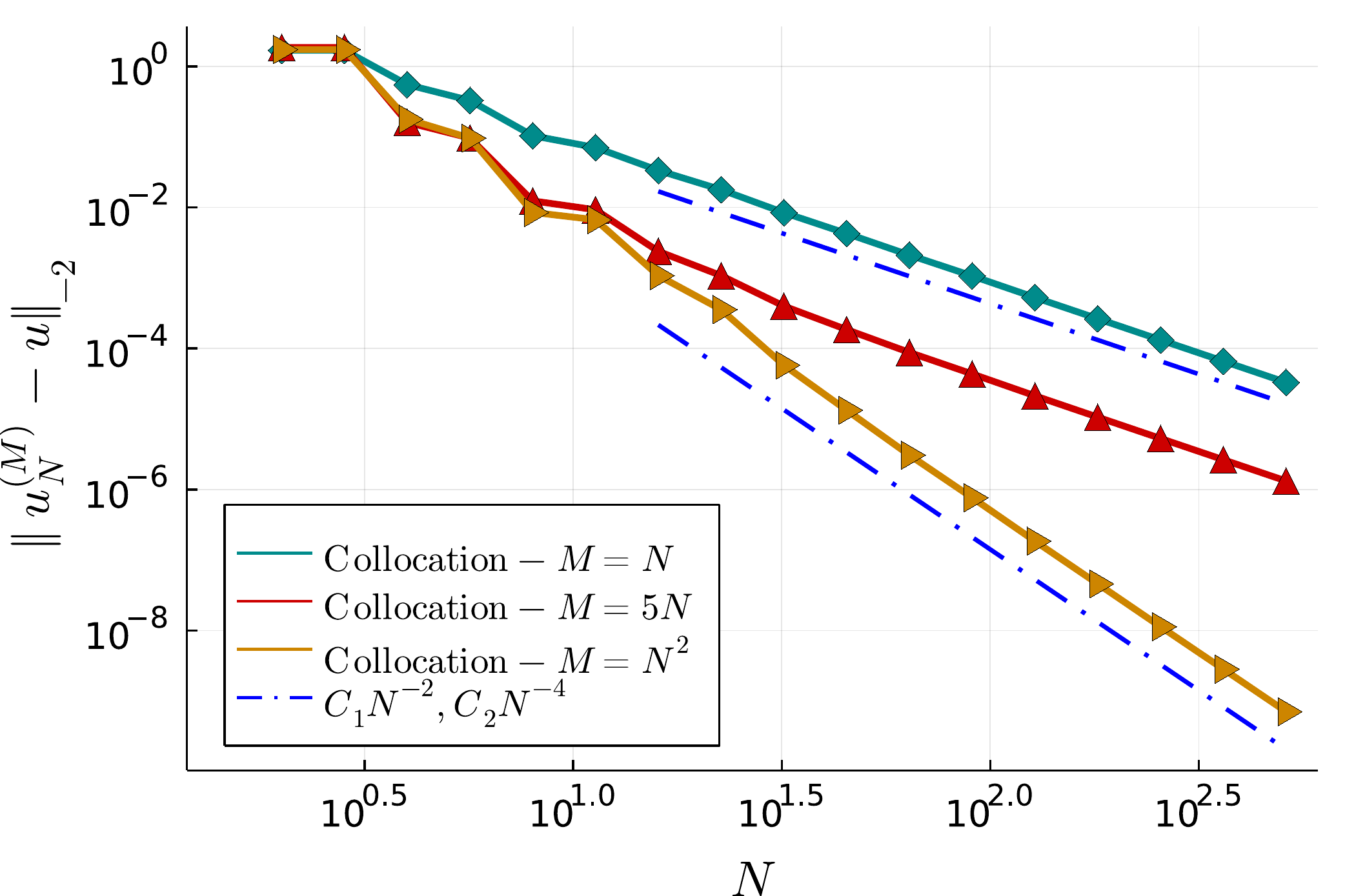}
				\caption{The approximation error in $H^{-2}$ for a range of rates of oversampling.}
				\label{fig:Sobolev_error_flow_example_superlinear_oversampling}
			\end{figure}
			
			We already mentioned in \S\ref{sec:previous_convergence_results} that for $N$ sufficiently large the estimate Eq.~\eqref{eqn:estimate_for_potential_flow_example} also predicts the reduction in error constant achieved by linear oversampling. In particular, if we take $M(N)=JN$ for some $J\in\mathbb{N}$, Eq.~\eqref{eqn:estimate_for_potential_flow_example} predicts that as long as $1\leq J\lesssim N$ the error should decay at quadratic rate in $J$. This feature is indeed observed in Fig.~\ref{fig:Sobolev_error_flow_example_linear_oversampling} and means that already linear oversampling is advantageous -- in the present example it reduces the error constant at quadratic rate whilst only incurring linear cost in $J$.
			
			\begin{figure}[h!]
				\centering\vspace{-0.1cm}
				\includegraphics[width=0.7\textwidth]{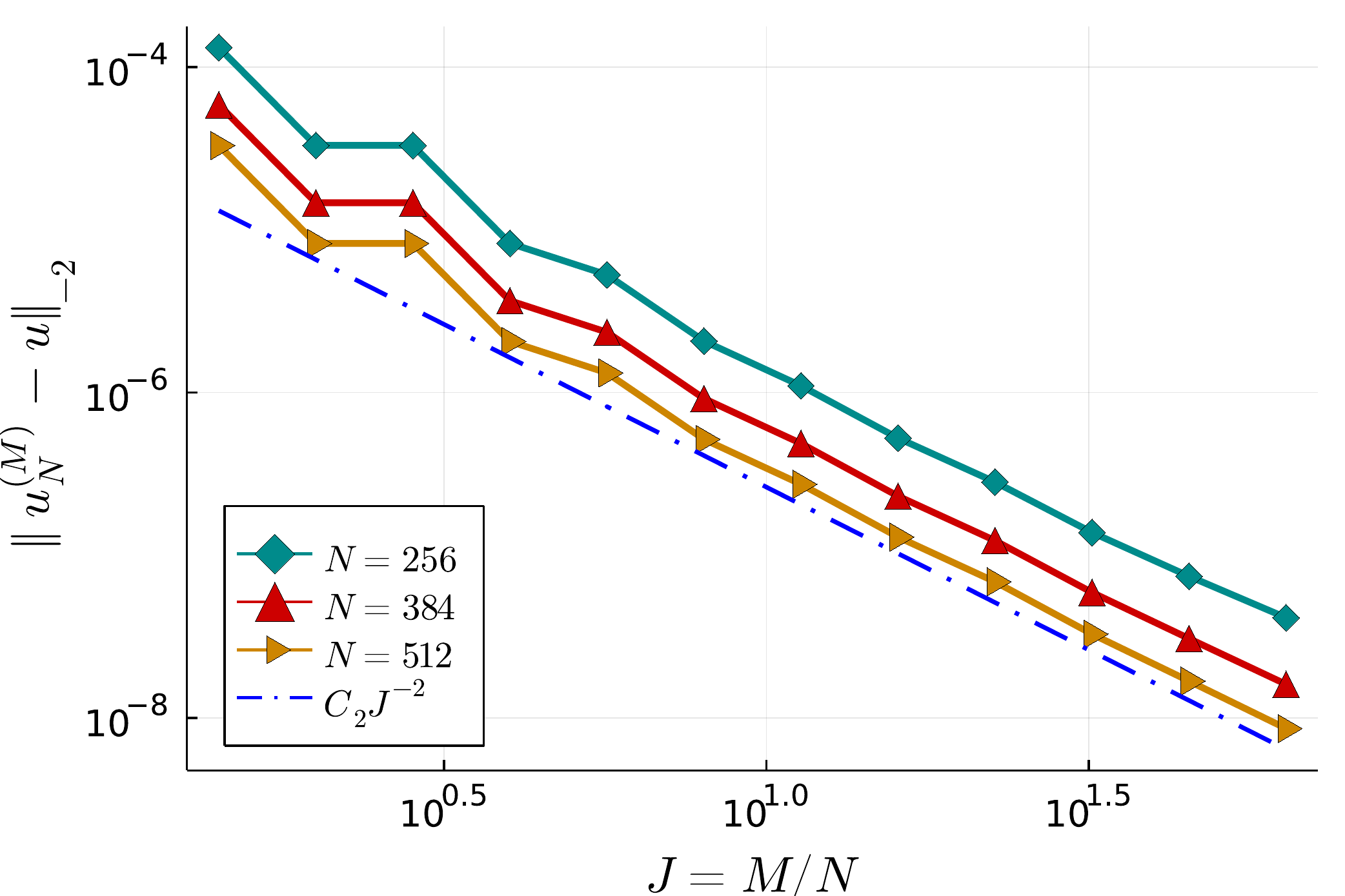}
				\caption{The approximation error in $H^{-2}$ as a function of the linear oversampling constant, $J=M/N$.}
				\label{fig:Sobolev_error_flow_example_linear_oversampling}\vspace{-0.1cm}
			\end{figure}
			
			\subsection{Application to stationary solutions of the heat equation}\label{sec:example_appl_stationary_solutions_heat_eqn}
			In our second example we look for a stationary solution of the heat equation on the interior of a compact domain $\Omega$ which in this case is the kite shape shown in Fig.~\ref{fig:obstacle_and_heat_field}. This means we seek a solution $\phi$, the temperature of the medium, which satisfies Eq.~\eqref{eqn:interior_dirichlet_Laplace} with, in our case, the boundary condition:
			\begin{align*}
				f(x)=\cos\left(5 x_1/\sqrt{2}+5x_2/\sqrt{2}\right).
			\end{align*}
			Therefore we can use Eq.~\eqref{eqn:single_layer_formulation_interior_Dirichlet} to solve for $\phi$.
			\begin{figure}[h!]
				\centering
				\includegraphics[width=0.625\textwidth]{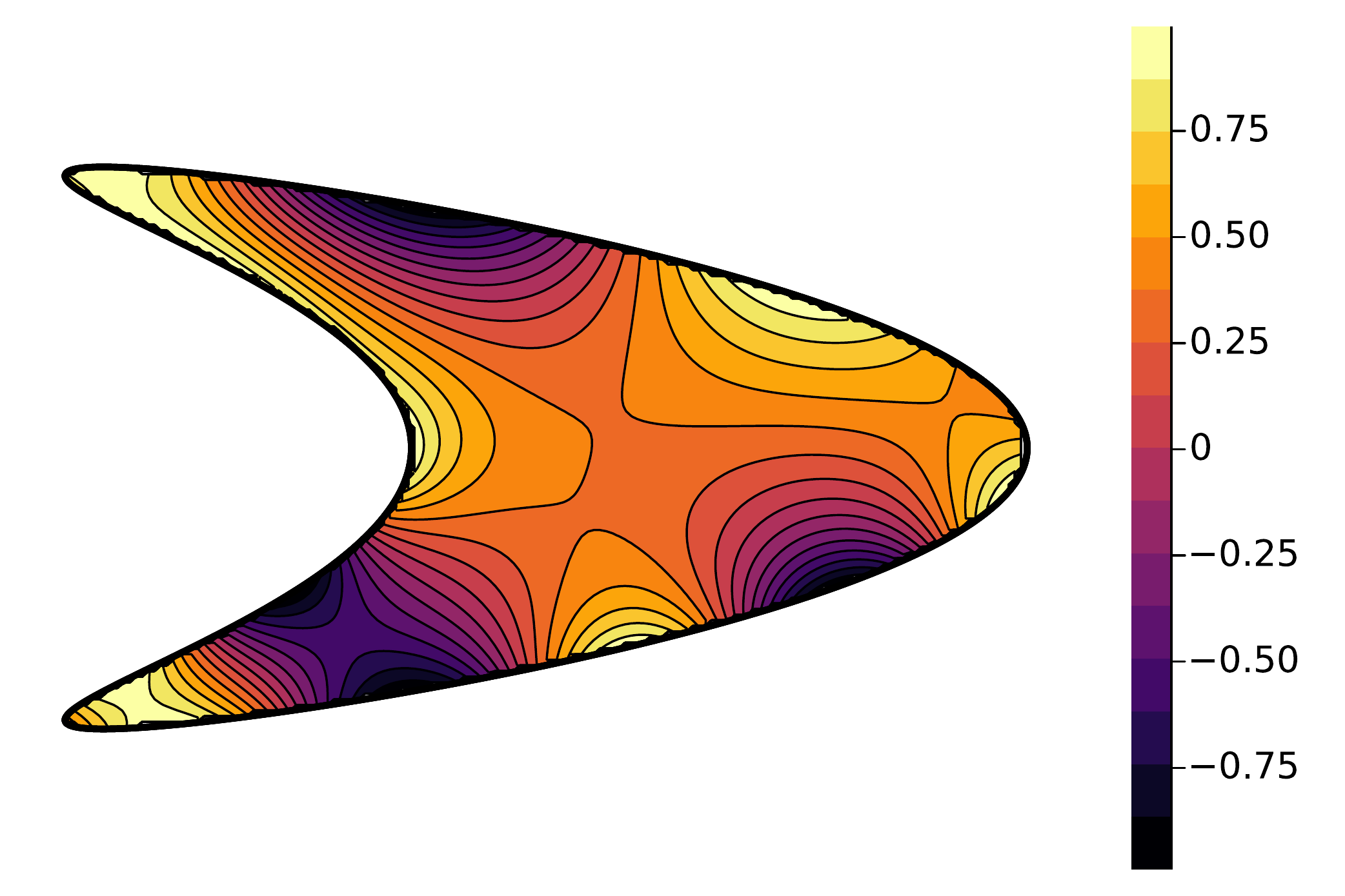}
				\caption{The stationary temperature field $\phi$ inside $\Omega$.}
				\label{fig:obstacle_and_heat_field}
			\end{figure}
			
			We again use a linear spline basis (i.e. $d=1$) and note that in this example we have $\alpha=-1/2$ for the integral operator $V=\mathcal{S}$. Thus the results of Thm.~\ref{thm:compact_perturbation_of_discrete_bubnov_galerkin} predict the following asymptotic convergence rate for $M,N$ sufficiently large:
			\begin{align}\label{eqn:estimate_for_stationary_heat_equation}
				\|u_N^{(M)}-u\|_{-4}\leq C (M^{-3}+N^{-6})\|u\|_2,
			\end{align}
			for some $C>0$. In Figs.~\ref{fig:Sobolev_error_heat_example_superlinear_oversampling} \& \ref{fig:Sobolev_error_heat_example_linear_oversampling} we plot the approximation error $\|u_N^{(M)}-u\|_{-4}$ for analogous amounts of oversampling as in Figs.~\ref{fig:Sobolev_error_flow_example_superlinear_oversampling} \& \ref{fig:Sobolev_error_flow_example_linear_oversampling} (but measured in $H^{-4}$ as opposed to $H^{-2}$). Again we included in the dash-dotted lines the predicted convergence rates. As before the constants $C_1,C_2,C_3$ in these curves are in no relation to Eq.~\eqref{eqn:estimate_for_stationary_heat_equation} and are simply included to make the plots easier to read.
			
			In Fig.~\ref{fig:Sobolev_error_heat_example_superlinear_oversampling} we observe that the predicted asymptotic rates of convergence ($\mathcal{O}(N^{-3})$ for linear oversampling, when $M(N)=N,5N$ and $\mathcal{O}(N^{-6})$ for quadratic oversampling, when $M(N)=N^2$) are indeed realised in practice.

			\begin{figure}[h!]
				\centering
				\includegraphics[width=0.7\textwidth]{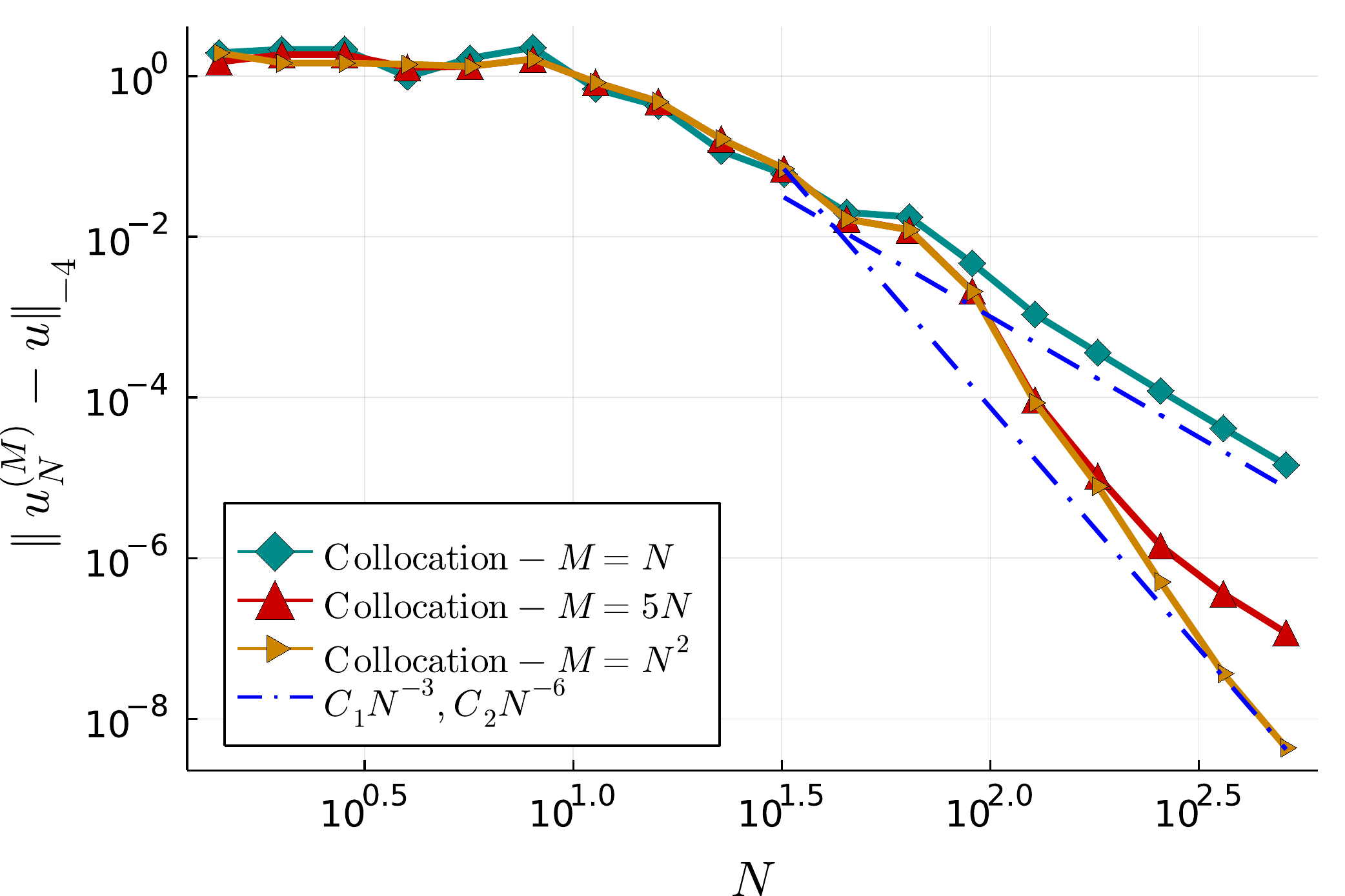}
				\caption{The approximation error in $H^{-4}$ for a range of rates of oversampling.}
				\label{fig:Sobolev_error_heat_example_superlinear_oversampling}\vspace{-0.3cm}
			\end{figure}
		
			Moreover, we can also see in this case the predicted reduction in error constant with linear oversampling, $M(N)=JN$, $J\in\mathbb{N}$. Indeed according to Eq.~\eqref{eqn:estimate_for_stationary_heat_equation}, for $N$ sufficiently large, the error constant should decay at cubic rate in $J$ in the regime $1\leq J\lesssim N$. This is indeed observed in Fig.~\ref{fig:Sobolev_error_heat_example_linear_oversampling}, where the levelling off of this convergence which can be seen for large values of $J$ in each curve is simply due to reaching the upper limit on $J$ for which this behaviour is predicted.
			
			\begin{figure}[h!]
				\centering
				\includegraphics[width=0.7\textwidth]{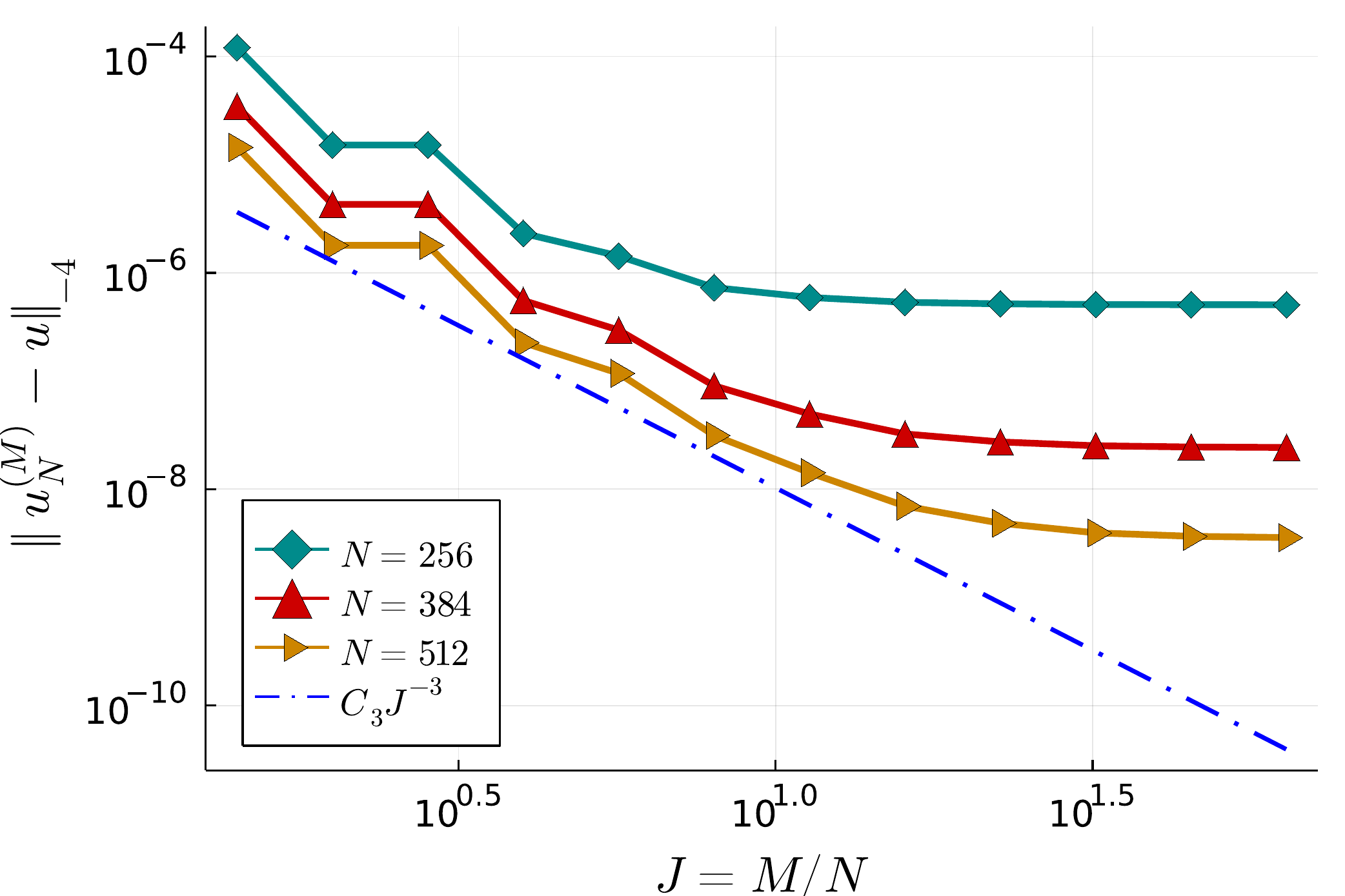}
				\caption{The approximation error in $H^{-4}$ as a function of the linear oversampling constant, $J=M/N$.}
				\label{fig:Sobolev_error_heat_example_linear_oversampling}\vspace{-0.3cm}
			\end{figure}

			\section{Conclusions}\label{sec:conclusions}
			In this manuscript we presented a novel argument that allowed us to understand the asymptotic convergence rates for oversampled collocation methods for integral operators which are certain compact perturbations of a pseudodifferential form. This provides a significant extension of prior results in \cite{maierhoferhuybrechs2021} and explains the advantageous convergence properties of the oversampled collocation method in several practically relevant settings. We provided multiple numerical examples to demonstrate these favorable properties in practice and these confirm the theoretical predictions of the convergence rates.
			
			Numerical experiments provided in \cite{maierhoferhuybrechs2021} suggest that the conclusions of Thm.~\ref{thm:compact_perturbation_of_discrete_bubnov_galerkin} also hold under weaker assumptions on the compact perturbation $\mathcal{K}$. Indeed, it was shown in \cite[\S4.G]{maierhoferthesis2021} that a related asymptotic convergence result holds true also for the single layer integral operator arising in the Dirichlet problem for the Helmholtz equation, in which case the perturbation satisfies a weaker continuity condition, namely that $\mathcal{K}:H^{p}\rightarrow H^{p+3}$ is continuous  for all $p\in\mathbb{R}$ \cite{colton2013integral}. Thus we expect that further extensions of this result are possible and these will be explored in future research.
			
			In addition to further study of collocation methods, we believe the framework of oversampling could also prove advantageous in improving convergence properties of Nystr\"om methods for Fredholm integral equations \cite[Chapter 12]{kress2014}. Nystr\"om methods can in many cases be regarded as a further level of discretisation from collocation methods, which suggests that some of the favourable features of oversampling may be inherited in this setting \cite{atkinson_1997}.
			
			\section*{Acknowledgements}
			The authors would like to thank Katharina Schratz (Sorbonne University) for valuable input on a draft of this manuscript and Simon Chandler-Wilde (University of Reading) for a number of interesting discussions and feedback on an earlier version of these results.
			
			GM gratefully acknowledges funding for this project from the European Research Council (ERC) under the European Union’s Horizon 2020 research and innovation programme (grant agreement No.\ 850941) as well as from the UK Engineering and Physical Sciences Research Council (EPSRC) grant EP/L016516/1 for the University of Cambridge Centre for Doctoral Training, the Cambridge Centre for Analysis.
			DH acknowledges financial support from KU Leuven grant C14/55/055.

\bibliographystyle{siam}      
\bibliography{mybibfile}

\begin{appendix}
\section{Proof of Proposition~\ref{prop:perturbed_orthogonality_conditions}}\label{app:proof_of_central_a_priori_estimate}
For completeness let us firstly recall the statement of the proposition.
\begin{proposition}\label{prop:perturbed_orthogonality_conditions_app}
	Suppose $a_N^{(M)}\in S_N$ satisfies 
	\begin{align}\label{eqn:unperturbed_apriori_orhtogonality_condition2}
		\left\langle V_0\chi_N,V_0a_N^{(M)}\right\rangle_M=\left\langle V_0\chi_N,V_0\tilde{v}\right\rangle_M+\epsilon(\chi_N,b-c_N),\ \quad\forall \chi_N\in S_N,
	\end{align}
	for some $\tilde{v},b\in H^{d+1}$ and a sequence of smoothest splines $c_N\in S_N$, $N\in\mathbb{N}$. Then there is a constant $C>0$ independent of $N,M,\tilde{v},b,c_N$ such that for all $M\geq N>0$:
	\begin{align*}
		\|\tilde{v}-a_N^{(M)}\|_{4\alpha-(d+1)}&\leq C \left(M^{2\alpha-(d+1)}+N^{4\alpha-2(d+1)}\right)\left(\|\tilde{v}\|_{d+1}+\|b\|_{d+1}\right)\\
		&\quad\quad +CN^{-1}\|b-c_N\|_{4\alpha-(d+1)}.
	\end{align*}
\end{proposition}
We will formulate the proof of Prop.~\ref{prop:perturbed_orthogonality_conditions} in a sequence of four lemmas whose proofs are given in sequence in \ref{sec:app_proof_of_lemmas}. To begin with, it is convenient to use the following basis for $S_N$ (cf. \cite[\S2]{chandler1990} and \cite[\S7]{sloan_1992}), where here we write $\Lambda_N=\left\{\mu\in\mathbb{Z}:-N/2<\mu\leq N/2\right\}$ and $\Lambda_N^*=\Lambda_N\setminus\{0\}$:
\begin{align*}
	\psi_\mu(x)=\begin{cases}
		1,&\mu=0,\\
		\sum_{k\equiv \mu (N)}(\mu/k)^{d+1} e^{2\pi i kx},&\mu\in\Lambda_N^*.
	\end{cases}
\end{align*}
The coefficients $a_\mu, \mu\in\Lambda_N$, of the expansion $a_N^{(M)}=\sum_{\mu\in\Lambda_N}a_\mu\psi_\mu$ in this basis are equal to the Fourier coefficients $\widehat{a_N^{(M)}}_\mu, \mu\in\Lambda_N$. In this basis the orthogonality conditions Eq.~\eqref{eqn:unperturbed_apriori_orhtogonality_condition2} are equivalent to the following $N\times N$ linear system
\begin{align}\label{eqn:discrete_orthogonality_special_basis}
	\sum_{\nu\in\Lambda_N} \left\langle V_0\psi_\mu,V_0\psi_{\nu}\right\rangle_M a_\nu=\left\langle V_0\psi_{\mu}, V_0\tilde{v}\right\rangle_M+\epsilon(\psi_{\mu},b-c_N), \quad \mu\in\Lambda_N.
\end{align}
Let us denote by $a_\mu',\mu\in\Lambda_N$ the solution of the following system
\begin{align}\label{eqn:app_unperturbed_linear_system}
	\sum_{\nu\in\Lambda_N} \left\langle V_0\psi_\mu,V_0\psi_{\nu}\right\rangle_M a_\nu'=\left\langle V_0\psi_{\mu}, V_0\tilde{v}\right\rangle_M, \quad \mu\in\Lambda_N,
\end{align}
which is precisely what was studied in \cite[\S3.3]{maierhoferhuybrechs2021}. Our first lemma allows us to estimate the error in approximation in terms of the coefficients $a_\mu,a_\mu'$.

\begin{lemma}\label{lem:app_estimate_sobolev_norm_in_terms_of_low_coef}
	There is $C>0$ such that, for any $M\geq N\geq1$,
	\begin{align}\label{eqn:app_estimate_sobolev_norm_in_terms_of_low_coef}
			\|a_N^{(M)}-\tilde{v}\|_{{4\alpha-(d+1)}}^2&\leq C\sum_{\mu\in\Lambda_N}\left[\mu\right]^{2(4\alpha-(d+1))}|{a}_\mu'-\hat{v}_\mu|^2+C\sum_{\mu\in\Lambda_N}\left[\mu\right]^{2(4\alpha-(d+1))}|{a}_\mu-a_\mu'|^2.
	\end{align}
\end{lemma}
The second lemma draws on results about the terms $|{a}_\mu'-\hat{v}_\mu|$ from \cite{maierhoferhuybrechs2021}.

\begin{lemma}\label{lem:app_unperturbed_a_prior_estimate_from_previous_paper}
	If $d>2\alpha$, there are constants $N_0,C>0$ such that for any $M\geq N\geq N_0$ we have
	\begin{align}\label{eqn:estimate_sobolev_error_from_spline_expansion_2}
		\sum_{\mu\in\Lambda_N}\left[\mu\right]^{2(4\alpha-(d+1))}|a_\mu'-\hat{v}_\mu|^2\leq C \|\tilde{v}\|_{d+1}^2\left(M^{4\alpha-2(d+1)}+N^{8\alpha-4(d+1)}\right).
	\end{align}
\end{lemma}
Given the above results it remains to bound the terms $|{a}_\mu-a_\mu'|$ from Eq.~\eqref{eqn:app_estimate_sobolev_norm_in_terms_of_low_coef}:
\begin{lemma}\label{lem:app_estimate_a-a'_in_terms_of_epsilon}
	We have, for any $M\geq N\geq 1$:
	\begin{align}\label{eqn:app_estimate_a-a'_in_terms_of_epsilon}
		|a_\mu-a_\mu'|\leq [\mu]^{-4\alpha}|\epsilon(\psi_\mu,b-c_N)|.
	\end{align}
\end{lemma}

Finally, we can use the continuity properties of $\mathcal{K}$, Eq.~\eqref{eqn:decay_of_K_coefficients}, to estimate the size of these terms:
\begin{lemma}\label{lem:app_estimate_size_of_epsilon_terms}
	If $d>2\alpha$ there are constants $N_0,C>0$ such that for any $M\geq N\geq N_0$
	\begin{align}\begin{split}\label{eqn:app_estimate_size_of_epsilon_terms}
			\sum_{\mu\in\Lambda_N}[\mu]^{-2(d+1)}|&\epsilon(\psi_\mu,b-c_N)|^2\\
			&\leq C N^{-2}\|b-c_N\|_{4\alpha-(d+1)}^2+C(M^{2\alpha-(d+1)}+N^{4\alpha-2(d+1)})^2\|b\|_{d+1}^2.
		\end{split}
	\end{align}
\end{lemma}

We can now combine Lemmas~\ref{lem:app_estimate_sobolev_norm_in_terms_of_low_coef}-\ref{lem:app_estimate_size_of_epsilon_terms} to complete the proof of Prop.~\ref{prop:perturbed_orthogonality_conditions}.
\begin{proof}[Proof of Prop.~\ref{prop:perturbed_orthogonality_conditions_app}]
	Combining Eq.~\eqref{eqn:app_estimate_size_of_epsilon_terms} with Eq.~\eqref{eqn:app_estimate_a-a'_in_terms_of_epsilon} and plugging this together with Eq.~\eqref{eqn:estimate_sobolev_error_from_spline_expansion_2} into \eqref{eqn:estimate_for_small_frequency_projection} yields
	\begin{align*}
		\|a_N^{(M)}-\tilde{v}\|_{4\alpha-(d+1)}&\leq C(M^{2\alpha-(d+1)}+N^{4\alpha-2(d+1)})\left(\|\tilde{v}\|_{d+1}+\|b\|_{d+1}\right)+CN^{-1}\|b-c_N\|_{4\alpha-(d+1)},
	\end{align*}
	for some constant $C>0$ independent of $a_N^{(M)},\tilde{v},N,M,b,c_N$, i.e. which completes the proof of Prop.~\ref{prop:perturbed_orthogonality_conditions}.
\end{proof}

\section{Proofs of Lemmas~\ref{lem:app_estimate_sobolev_norm_in_terms_of_low_coef}-\ref{lem:app_estimate_size_of_epsilon_terms}}\label{sec:app_proof_of_lemmas}
\begin{proof}[Proof of Lemma~\ref{lem:app_estimate_sobolev_norm_in_terms_of_low_coef}]
	We follow the proof of \cite[Thm.~3]{maierhoferhuybrechs2021} and introduce the `projection' $P_N$ onto the low-frequencies $P_N:f\mapsto\sum_{\mu\in\Lambda_N}\hat{f}_\mu \psi_\mu(x)$ for which there is $C_p$ for any $p<d+1/2$ such that for any $p+1/2<q\leq d+1$ and $f\in H^q$ we have \cite[cf. \S8]{saranen2013periodic}
	\begin{align}\label{eqn:estimate_for_small_frequency_projection}
		{\|f-P_Nf\|_{{p}}\leq C_p N^{p-q}\|f\|_{{q}}},\quad \forall p+1/2<q\leq d+1.
	\end{align}
	Moreover, as shown in \cite{maierhoferhuybrechs2021}, we have
	\begin{align*}
		\|u_N^{(M)}-P_N\tilde{v}\|_{{4\alpha-(d+1)}}^2&=\sum_{\mu\in\Lambda_N}\left[\mu\right]^{2(4\alpha-(d+1))}|{a}_\mu-\hat{v}_\mu|^2\left(1+\sum_{l\neq0}\left[\frac{\mu}{\mu+lN}\right]^{4(d+1)-8\alpha}\right)\\
		&\quad\leq \tilde{C}\sum_{\mu\in\Lambda_N}\left[\mu\right]^{2(4\alpha-(d+1))}|{a}_\mu'-\hat{v}_\mu|^2+\sum_{\mu\in\Lambda_N}\left[\mu\right]^{2(4\alpha-(d+1))}|{a}_\mu-a_\mu'|^2,
	\end{align*}
	for some constant $\tilde{C}>0$ independent of $v,M,N$, where $\hat{v}_n$ are the Fourier coefficients of $\tilde{v}$.
\end{proof}

\begin{proof}[Proof of Lemma~\ref{lem:app_unperturbed_a_prior_estimate_from_previous_paper}]
	Thee proof of this statement is given in \cite[Proof of Thm.~3]{maierhoferhuybrechs2021}.
\end{proof}
\begin{proof}[Proof of Lemma~\ref{lem:app_estimate_a-a'_in_terms_of_epsilon}]
	As observed in \cite[Appendix~C]{maierhoferhuybrechs2021} the $N\times N$ matrix in the linear system Eq.~\eqref{eqn:discrete_orthogonality_special_basis} is diagonal:
	\begin{align*}
		\left\langle V_0\psi_{\mu}, V_0\psi_\nu\right\rangle_M=\begin{cases}0,&\text{if\ } \mu\neq \nu,\\
			1,&\text{if\ }\mu=\nu=0,\\
			[\mu]^{4\alpha}\frac{1}{J}\sum_{j=1}^J\left|1+\Omega\left(\frac{j}{J},\frac{\mu}{N}\right)\right|^2,&\text{if\ }\mu=\nu\neq 0,
		\end{cases}
	\end{align*}
	where $J=M/N\in\mathbb{N}$, and
	\begin{align*}
		\Omega(\xi,y)&=|y|^{{d+1}-2\alpha}\sum_{l\neq 0}\frac{1}{|l+y|^{{d+1}-2\alpha}}e^{2\pi il\xi}.
	\end{align*}
	Thus we have
	\begin{align*}
		a_\mu-a_\mu'=[\mu]^{-4\alpha}D\left(\frac{\mu}{N}\right)^{-1}\epsilon(\psi_\mu,b-c_N),\quad \mu\in\Lambda_N.
	\end{align*}
	We recall from Eq.~(52) in \cite{maierhoferhuybrechs2021} that
	\begin{align*}
		\frac{1}{J}\sum_{j=1}^J\left|1+\Omega\left(\frac{j}{J},\frac{\mu}{N}\right)\right|^2\geq 1,\quad \forall \mu\in\Lambda_N,
	\end{align*}
	whence Eq.~\eqref{eqn:app_estimate_a-a'_in_terms_of_epsilon} immediately follows.
\end{proof}

\begin{proof}[Proof of Lemma~\ref{lem:app_estimate_size_of_epsilon_terms}] We begin by expressing the terms $\epsilon(\psi_\mu,b-c_N)$ in a more explicit way: Let us write $\hat{b}_n,\hat{c}_n$ for the Fourier coefficients of $b,c_N$ respectively, noting that for all $n\in\Lambda_N^*,l\in\mathbb{Z},$ we have $\hat{c}_{n+lN}=n^{d+1}/(n+lN)^{d+1}\hat{c}_{n}$. Let us focus on the case $\mu\in\Lambda_N^*$ first, and consider $\mu=0$ after this initial calculation. We can compute
	\begin{align*}
		\mathcal{K}V_0\psi_\mu&=\sum_{n\in\mathbb{Z}} \sum_{m\equiv \mu(N)}[m]^{2\alpha} \left(\frac{\mu}{m}\right)^{d+1}k_{mn}e^{2\pi i nx},
	\end{align*}
hence
	\begin{align*}
		\langle \mathcal{K}V_0\psi_{\mu},&V_0(b-c_N)\rangle_M\\
		&=\sum_{n\in\mathbb{Z}} \sum_{m\equiv \mu(N)} \sum_{p\in\mathbb{Z}}[m]^{2\alpha} \left(\frac{\mu}{m}\right)^{d+1}\overline{k_{mn}}\,[p]^{2\alpha}(\hat{b}_p-\hat{c}_p)\left\langle\exp\left(2\pi i n\,\cdot\,\right),\exp\left(2\pi i p\,\cdot\,\right)\right\rangle_M\\
		&=\sum_{n\in\mathbb{Z}} \sum_{m\equiv \mu(N)} \sum_{p\equiv n(M)}[m]^{2\alpha} \left(\frac{\mu}{m}\right)^{d+1}\overline{k_{mn}}\,[p]^{2\alpha}(\hat{b}_p-\hat{c}_p).
	\end{align*}
	Similarly we find
	\begin{align*}
		\langle V_0\psi_{\mu},&\mathcal{K}^* V_0(b-c_N)\rangle_M\\
		&=\sum_{n\in\mathbb{Z}}\sum_{m\in\mathbb{Z}}[m]^{2\alpha}(\hat{b}_m-\hat{c}_m)\overline{k_{nm}}\left\langle \sum_{p\equiv\mu(N)}[p]^{2\alpha}\left(\frac{\mu}{p}\right)^{d+1}\exp\left(2\pi ip\, \cdot\,\right),\exp\left(2\pi in\, \cdot\,\right)\right\rangle_M\\
		&=\sum_{p\equiv\mu(N)}\sum_{n\equiv p(M)}\sum_{m\in\mathbb{Z}}[p]^{2\alpha}\left(\frac{\mu}{p}\right)^{d+1}[m]^{2\alpha}(\hat{b}_m-\hat{c}_m)\overline{k_{nm}}.
	\end{align*}
	Therefore we have after relabelling the dummy variables in the sums:
	\begin{align*}
		\epsilon(\psi_\mu,b-c_N)&=\left\langle V_0\psi_{\mu},\mathcal{K}^* V_0(b-c_N)\right\rangle_M-\left\langle \mathcal{K}V_0\psi_{\mu},V_0(b-c_N)\right\rangle_M\\
		&=\sum_{p\equiv\mu(N)}\sum_{n\equiv p(M)}\sum_{m\in\mathbb{Z}}[p]^{2\alpha}\left(\frac{\mu}{p}\right)^{d+1}[m]^{2\alpha}(\hat{b}_m-\hat{c}_m)\overline{k_{nm}}\\
		&\quad\quad-\sum_{m\in\mathbb{Z}} \sum_{n\equiv \mu(N)} \sum_{p\equiv m(M)}[n]^{2\alpha} \left(\frac{\mu}{n}\right)^{d+1}\overline{k_{nm}}\,[p]^{2\alpha}(\hat{b}_p-\hat{c}_p).
	\end{align*}
	We can then extract the low-frequency terms in both sums,
	\begin{align*}
		\epsilon(\psi_\mu,b-c_N)	&=\sum_{p\equiv\mu(N)}\sum_{m\in\mathbb{Z}}[p]^{2\alpha}\left(\frac{\mu}{p}\right)^{d+1}[m]^{2\alpha}(\hat{b}_m-\hat{c}_m)\overline{k_{pm}}\\
		&\quad\quad\quad-\sum_{m\in\mathbb{Z}} \sum_{n\equiv \mu(N)} [n]^{2\alpha} \left(\frac{\mu}{n}\right)^{d+1}\overline{k_{nm}}\,[m]^{2\alpha}(\hat{b}_m-\hat{c}_m)\\
		&\quad\quad+\sum_{p\equiv\mu(N)}\sum_{n\equiv p(M),n\neq p}\sum_{m\in\mathbb{Z}}[p]^{2\alpha}\left(\frac{\mu}{p}\right)^{d+1}[m]^{2\alpha}(\hat{b}_m-\hat{c}_m)\overline{k_{nm}}\\
		&\quad\quad\quad-\sum_{m\in\mathbb{Z}} \sum_{n\equiv \mu(N)} \sum_{\substack{p\equiv m(M)\\p\neq m}}[n]^{2\alpha} \left(\frac{\mu}{n}\right)^{d+1}\overline{k_{nm}}\,[p]^{2\alpha}(\hat{b}_p-\hat{c}_p),
	\end{align*}
	which are found to cancel:
	\begin{align*}
		\epsilon(\psi_\mu,b-c_N)&=\underbrace{\sum_{p\equiv\mu(N)}\sum_{n\equiv p(M),n\neq p}\sum_{m\in\mathbb{Z}}[p]^{2\alpha}\left(\frac{\mu}{p}\right)^{d+1}[m]^{2\alpha}(\hat{b}_m-\hat{c}_m)\overline{k_{nm}}}_{=:A_1}\\
		&\quad\quad-\underbrace{\sum_{m\in\mathbb{Z}} \sum_{n\equiv \mu(N)} \sum_{\substack{p\equiv m(M)\\p\neq m}}[n]^{2\alpha} \left(\frac{\mu}{n}\right)^{d+1}\overline{k_{nm}}\,[p]^{2\alpha}(\hat{b}_p-\hat{c}_p)}_{=:A_2}.
	\end{align*}
	In what follows we will bound the remaining two terms $A_1,A_2$ individually. For this (and the remainder of this appendix) we shall make use of the notation $\lesssim$ to indicate an implicit constant in the inequality, which is in all cases independent of $a_N^{(M)},\tilde{v},b,c_N,N,M$, though it may sometimes depend on other parameters in the inequalities. Where this is of relevance we will indicate this dependence by a subscript, for instance $\lesssim_r$.
	\begin{align}\nonumber
		\left|A_1\right|&\leq\left|\sum_{n\equiv \mu(M),n\neq \mu}[\mu]^{2\alpha}\widehat{\left(\mathcal{K}V_0(b-c_N)\right)}_n\right|+ \left|\sum_{p\equiv\mu(N),p\neq \mu}\sum_{n\equiv p(M),n\neq p}[p]^{2\alpha}\left(\frac{\mu}{p}\right)^{d+1}\widehat{\left(\mathcal{K}V_0(b-c_N)\right)}_n\right|\\\nonumber
		&\lesssim_r[\mu]^{2\alpha}M^{-r}\|\mathcal{K}V_0(b-c_N)\|_r+ \sum_{l\neq 0}\sum_{\substack{n\equiv \mu+lN(M)\\n\neq \mu+lN}}[\mu+lN]^{2\alpha}\left|\frac{\mu}{\mu+lN}\right|^{d+1}\left|\widehat{\left(\mathcal{K}V_0(b-c_N)\right)_n}\right|\\\nonumber
		&\lesssim_r[\mu]^{2\alpha}M^{-r}\|\mathcal{K}V_0(b-c_N)\|_r+\sum_{l\neq 0}[\mu+lN]^{2\alpha}\left|\frac{\mu}{\mu+lN}\right|^{d+1}\sum_{n\equiv \mu+lN(M)}\left|\widehat{\left(\mathcal{K}V_0(b-c_N)\right)_n}\right|\end{align}
	for any $r>1/2$. Therefore,\vspace{-0.1cm}\begin{align}
		\begin{split}\label{eqn:upper_bound_on_A_1}
			\left|A_1\right|&\lesssim_r[\mu]^{2\alpha}M^{-r}\|\mathcal{K}V_0(b-c_N)\|_r+\sum_{l\neq 0}[\mu+lN]^{2\alpha}\left|\frac{\mu}{\mu+lN}\right|^{d+1}C_r\|\mathcal{K}V_0(b-c_N)\|_r,
		\end{split}
	\end{align}
	for any $r>1/2$. For the second term we have\\\vspace{-0.5cm}
	\begin{align*}
		\left|A_2\right|
		&=\Bigg|\sum_{n\equiv \mu(N)}[n]^{2\alpha} \left(\frac{\mu}{n}\right)^{d+1} \sum_{m\in\mathbb{Z}} \sum_{\substack{p\equiv m(M)\\p\neq m}}\overline{k_{nm}}\widehat{\left(V_0(b-c_N)\right)}_p\Bigg|\\
		&\leq \underbrace{\Bigg|\sum_{n\equiv \mu(N)}[n]^{2\alpha} \left(\frac{\mu}{n}\right)^{d+1} \sum_{m\in \Lambda_N} \sum_{\substack{p\equiv m(M)\\p\neq m}}\overline{k_{nm}}\widehat{\left(V_0(b-c_N)\right)}_p\Bigg|}_{=:A_{21}}\\
		&\quad\quad+\underbrace{\Bigg|\sum_{n\equiv \mu(N)}[n]^{2\alpha} \left(\frac{\mu}{n}\right)^{d+1} \sum_{m\notin\Lambda_N} \sum_{\substack{p\equiv m(M)\\p\neq m,p\notin\Lambda_N}}\overline{k_{nm}}\widehat{\left(V_0(b-c_N)\right)}_p\Bigg|}_{=:A_{22}}.
	\end{align*}\vspace{-0.5cm}\\
	Let us estimate $A_{21},A_{22}$ separately: Firstly, using the continuity properties of $\mathcal{K}$, i.e.\ Eq.~\eqref{eqn:decay_of_K_coefficients} we find for any $p_1,q_1\in\mathbb{R}$:\\\vspace{-0.5cm}
	\begin{align*}
		A_{21}&\leq \sum_{n\equiv \mu(N)}[n]^{2\alpha} \left|\frac{\mu}{n}\right|^{d+1} \sum_{m\in \Lambda_N} \sum_{\substack{p\equiv m(M)\\p\neq m}}|\overline{k_{nm}}|\left|\widehat{\left(V_0(b-c_N)\right)}_p\right|\\
		&\lesssim_{p_1,q_1}\sum_{n\equiv \mu(N)}[n]^{2\alpha} \left|\frac{\mu}{n}\right|^{d+1} \sum_{m\in \Lambda_N} \sum_{\substack{p\equiv m(M)\\p\neq m}}(1+|n|)^{-p_1}(1+|m|)^{-q_1}\left|\widehat{\left(V_0(b-c_N)\right)}_p\right|.
	\end{align*}\vspace{-0.5cm}\\
	Moreover, we have for $q_1>-4\alpha+d+3/2$:\\\vspace{-0.25cm}
	\begin{align}\nonumber
		\sum_{m\in \Lambda_N}(1+|m|)^{-q_1} &\sum_{\substack{p\equiv m(M)\\p\neq m}}[p]^{2\alpha}\left|\widehat{\left(b-c_N\right)}_p\right|\\\nonumber
		&=\sum_{m\in \Lambda_N}(1+|m|)^{-q_1} \sum_{l\neq 0}[m+lM]^{2\alpha}\left|\hat{b}_{m+lM}-\left(\frac{m}{m+lM}\right)^{d+1}\hat{c}_{m}\right|\\\nonumber
		&\leq \sum_{m\in \Lambda_N}(1+|m|)^{-q_1} \sum_{l\neq 0}[m+lM]^{2\alpha}\left(|\hat{b}_{m+lM}|+|\hat{c}_m|\left|\frac{m}{m+lM}\right|^{d+1}\right)\\
		\begin{split}\label{eqn:using_d+1>1}
			&\lesssim\sum_{m\in \Lambda_N}(1+|m|)^{-q_1} \sum_{l\neq 0}[m+lM]^{2\alpha-(d+1)}[m+lM]^{d+1}|\hat{b}_{m+lM}|\\
			&\quad\quad+M^{2\alpha-(d+1)}\sum_{m\in \Lambda_N}(1+|m|)^{-q_1}{|m|}^{d+1}|\hat{c}_m|.\end{split}\end{align}\vspace{-0.4cm}\\
	Therefore, for $q_1>-4\alpha+d+3/2$,\vspace{-0.3cm}\\
	\begin{align}\nonumber
		\sum_{m\in \Lambda_N}&(1+|m|)^{-q_1} \sum_{\substack{p\equiv m(M)\\p\neq m}}[p]^{2\alpha}\left|\widehat{\left(b-c_N\right)}_p\right|	\\\nonumber&\lesssim_{q_1}\left(\sum_{m\in \Lambda_N}(1+|m|)^{-q_1} \sum_{l\neq 0}[m+lM]^{4\alpha-2(d+1)}\right)^{\frac{1}{2}}\|b\|_{d+1}+M^{2\alpha-(d+1)}\|c_N\|_{4\alpha-(d+1)}\\\label{eqn:using_m_in_Lambda_N}
		&\lesssim_{q_1} M^{2\alpha-(d+1)}\left(\|b\|_{d+1}+\|c_N\|_{4\alpha-(d+1)}\right)\\\nonumber
		&\lesssim_{q_1}  M^{2\alpha-(d+1)}\left(\|b\|_{d+1}+\|b-c_N\|_{4\alpha-(d+1)}+\|b\|_{4\alpha-(d+1)}\right)\\\nonumber
		&\lesssim_{q_1}  M^{2\alpha-(d+1)}\left(2\|b\|_{d+1}+\|b-c_N\|_{4\alpha-(d+1)}\right),
	\end{align}
	where in Eqs.~\eqref{eqn:using_d+1>1} \& \eqref{eqn:using_m_in_Lambda_N} we used that $|m|\leq N/2\leq M/2$ for all $m\in\Lambda_N$ and the consistency condition $d>2\alpha$, which implies that $d+1-2\alpha>1$. In the final line we also relied on the consistency condition $d>2\alpha$, which implies $d+1>4\alpha-(d+1)$. Therefore, we found, for $q_1>-4\alpha+d+3/2$:
	\begin{align}\begin{split}\label{eqn:upper_bound_on_A_21}
			A_{21}&\lesssim_{p_1,q_1}\left(\sum_{n\equiv \mu(N)}[n]^{2\alpha} \left|\frac{\mu}{n}\right|^{d+1}(1+|n|)^{-p_1}\right)M^{2\alpha-(d+1)}\left(2\|b\|_{d+1}+\|b-c_N\|_{4\alpha-(d+1)}\right)\end{split}
	\end{align}
	For $A_{22}$ we have by Eq.~\eqref{eqn:decay_of_K_coefficients}
	\begin{align*}
		&A_{22}\leq \sum_{n\equiv \mu(N)}[n]^{2\alpha} \left|\frac{\mu}{n}\right|^{d+1} \sum_{m\notin\Lambda_N} \sum_{\substack{p\equiv m(M)\\p\neq m}}|{\overline{k_{nm}}}||\widehat{\left(V_0(b-c_N)\right)}_p|.\end{align*}Thus\begin{align*}
		A_{22}&\lesssim_{s_2,t_2}\sum_{n\equiv \mu(N)}[n]^{2\alpha} \left|\frac{\mu}{n}\right|^{d+1} \sum_{m\notin\Lambda_N} \sum_{\substack{p\equiv m(M)\\p\neq m}}(1+|n|)^{-s_2}(1+|m|)^{-t_2}|\widehat{\left(V_0(b-c_N)\right)}_p|\\
		&\lesssim_{s_2,t_2}\sum_{n\equiv \mu(N)}[n]^{2\alpha} \left|\frac{\mu}{n}\right|^{d+1} (1+|n|)^{-s_2}\sum_{m\notin\Lambda_N} (1+|m|)^{-t_2}\sum_{p\in\mathbb{Z}}|\widehat{\left(V_0(b-c_N)\right)}_p|\\
		&\lesssim_{s_2,t_2,\delta}\hspace{-0.2cm}\sum_{n\equiv \mu(N)}[n]^{2\alpha} \left|\frac{\mu}{n}\right|^{d+1} (1+|n|)^{-s_2}\sum_{m\notin\Lambda_N} (1+|m|)^{-t_2}\|V_0(b-c_N)\|_{-2\alpha +d+1/2-\delta},
	\end{align*}
	where by the consistency assumption $d>2\alpha$ we were able to choose $\delta$ with $0<\delta<d-2\alpha$ which implies $\sum_{p\in\mathbb{Z}}[p]^{2(-2\alpha +d+1/2-\delta)}<\infty$. By the approximation property of smoothest splines (Assumption~1 in \cite{maierhoferhuybrechs2021}) we can choose $\chi_N\in S_N$ such that for some $c>0$ and for all $t<d+1/2$:
	\begin{align*}
		\|b-\chi_N\|_{t}\leq c\|b\|_{d+1}.
	\end{align*}
	This allows us to estimate, using the inverse property of {smoothest} splines (Assumption~2 in \cite{maierhoferhuybrechs2021}),
	\begin{align*}
		\|V_0(b-c_N)\|_{-2\alpha+d+1/2-\delta}&\lesssim \|b-c_N\|_{d+1/2-\delta}\\
		&\lesssim \left(\|b-\chi_N\|_{d+1/2-\delta}+\|\chi_N-c_N\|_{d+1/2-\delta}\right)\\
		&\lesssim \left(N^{-1/2-\delta}\|b\|_{d+1}+N^{2d+3/2-\delta-4\alpha}\|\chi_N-c_N\|_{4\alpha-(d+1)}\right)\\
		&\lesssim \left(N^{-1/2-\delta}\|b\|_{d+1}+N^{2d+3/2-\delta-4\alpha}\|b-\chi_N\|_{4\alpha-(d+1)}\right.\\
		&\quad\quad\quad\quad\quad\quad\quad\quad\quad\quad\quad\quad\left.+N^{2d+3/2-\delta-4\alpha}\|b-c_N\|_{4\alpha-(d+1)}\right)\\
		&\lesssim \left(N^{-1/2-\delta}\|b\|_{d+1}+N^{2d+3/2-\delta-4\alpha}\|b-c_N\|_{4\alpha-(d+1)}\right).
	\end{align*}
	Furthermore we have, whenever $t_2>1$,
	\begin{align*}
		\sum_{m\notin\Lambda_N} (1+|m|)^{-t_2}\leq \tilde{\tilde{C}}_{t_2} N^{1-t_2}
	\end{align*}
	for some constant $\tilde{\tilde{C}}_{t_2}>0$ independent of $N$. We can combine these estimates to show that, when $t_2=3/2-\delta+2(d+1)-4\alpha$,
	\begin{align}\label{eqn:upper_bound_on_A_22}
			A_{22}\lesssim_{s_2,t_2}\sum_{n\equiv \mu(N)}[n]^{2\alpha} &\left|\frac{\mu}{n}\right|^{d+1} (1+|n|)^{-s_2}\left(N^{4\alpha-2(d+1)}\|b\|_{d+1}+N^{-1}\|b-c_N\|_{4\alpha-(d+1)}\right).
	\end{align}
	This means combining Eqs.~\eqref{eqn:upper_bound_on_A_1}, \eqref{eqn:upper_bound_on_A_21} \& \eqref{eqn:upper_bound_on_A_22} gives a bound on $\epsilon(\psi_\mu,b-c_N)$, whenever $\mu\in\Lambda_N^*$. A similar bound for $\mu=0$ can be found by simply replacing the sums over $\sum_{n\equiv\mu(N)}$ by the unique choice $n=0$ with no summation. 
	
	\newpage This allows us to derive the upper bound given in Eq.~\eqref{eqn:app_estimate_size_of_epsilon_terms}: We have from the estimates derived in
	Eqs.~\eqref{eqn:upper_bound_on_A_1}, \eqref{eqn:upper_bound_on_A_21} \& \eqref{eqn:upper_bound_on_A_22}, that there are constants $C_r,C_q>0$ such that for any $r>1/2$ and any $q>0$
	\begin{align*}
		&C\sum_{\mu\in\Lambda_N}[\mu]^{-2(d+1)}|\epsilon(\psi_\mu,b-c_N)|^2=\sum_{\mu\in\Lambda_N}[\mu]^{-2(d+1)}|a_\mu-a'_\mu|^2\\&\leq C_r\sum_{\mu\in\Lambda_N}\left([\mu]^{2\alpha-(d+1)}M^{-r}+\sum_{l\neq 0}[\mu+lN]^{2\alpha-(d+1)}\right)^2\|\mathcal{K}V_0(b-c_N)\|_r^2\\
		&\quad+C_q\sum_{\mu\in\Lambda_N}\left(\sum_{n\equiv \mu(N)}[n]^{2\alpha-(d+1)}(1+|n|)^{-s}\right)^2\left(M^{2\alpha-(d+1)}\|b\|_{d+1}\right.\\
		&\quad\quad\left.+M^{2\alpha-(d+1)}\|b-c_N\|_{4\alpha-(d+1)}+N^{4\alpha-2(d+1)}\|b\|_{d+1}+N^{-1}\|b-c_N\|_{4\alpha-(d+1)}\right)^2.
	\end{align*}
	Given that $2\alpha-d<0$, we have $\sum_{l\neq 0}[\mu+lN]^{2\alpha-(d+1)}\lesssim N^{-1}$. Moreover $\mathcal{K}:H^{2\alpha-(d+1)}\rightarrow H^{r}$ is continuous, and so we find by taking $r,s>1$:
	\begin{align*}
		\sum_{\mu\in\Lambda_N}[\mu]^{-2(d+1)}|\epsilon(\psi_\mu,b-c_N)|^2&\lesssim N^{-2}\|b-c_N\|_{4\alpha-(d+1)}^2+(M^{2\alpha-(d+1)}+N^{4\alpha-2(d+1)})^2\|b\|_{d+1}^2.
	\end{align*}
\end{proof}
\end{appendix}
\end{document}